\documentclass[12pt,reqno]{amsart}
\usepackage{}
\usepackage{a4wide}
\numberwithin{equation}{section}
\usepackage{mathrsfs}
\usepackage{amsfonts}
\usepackage{amsmath,extpfeil}
\usepackage{stmaryrd}
\usepackage{amssymb}
\usepackage{amsthm}
\usepackage{mathrsfs}
\usepackage{url}
\usepackage{amsfonts}
\usepackage{amscd}
\usepackage{indentfirst}
\usepackage{enumerate}
\usepackage{amsmath,amsfonts,amssymb,amsthm}
\usepackage{amsmath,amssymb,amsthm,amscd}
\usepackage{graphicx,mathrsfs}
\usepackage{appendix}
\usepackage[numbers,sort&compress]{natbib}
\usepackage{color}
\usepackage[colorlinks, linkcolor=blue, citecolor=blue]{hyperref}
\usepackage{lineno,hyperref}
\usepackage{colortbl}
\usepackage{amsmath,amssymb,amsfonts}
\usepackage{algorithmic}
\usepackage{graphicx}
\usepackage{float}

\allowdisplaybreaks[4]

\newtheorem{Thm}{Theorem}[section]
\newtheorem{Lem}[Thm]{Lemma}

\newtheorem{Def}[Thm]{Definition}
\newtheorem{Rem}[Thm]{Remark}

\newcommand{\e}\varepsilon
\newcommand{\vs}{\vspace}

\begin{document}

\title[Critical Neumann problem in the half-space]
{Existence  of solutions for critical Neumann problem with superlinear perturbation in the half-space
}
\author[ Y. Deng, L. Shi and X. Zhang]{Yinbin Deng, Longge Shi and Xinyue Zhang}
\footnote{The research was supported by the Natural
	Science Foundation of China (No: 12271196, 11931012).
%
}
\address[Yinbin Deng]{School of Mathematics and Statistics and Key Laboratory of Nonlinear Analysis $\&$ Applications, Central China Normal University, Wuhan 430079, China}
\email{ybdeng@ccnu.edu.cn}

\address[Longge Shi]{School of Mathematics and Statistics and Key Laboratory of Nonlinear Analysis $\&$ Applications, Central China Normal University, Wuhan 430079, China}
\email{shilongge@mails.ccnu.edu.cn}

\address[Xinyue Zhang]{School of Mathematics and Statistics and Key Laboratory of Nonlinear Analysis $\&$  Applications, Central China Normal University, Wuhan 430079, China}
\email{xinyuezhang@mails.ccnu.edu.cn}

\date{\today}

\begin{abstract}
In this paper, we consider the existence and multiplicity of solutions for  the critical Neumann problem
\begin{equation}\label{1.1ab}
	\left\{
	\begin{aligned}
		-\Delta {u}-\frac{1}{2}(x \cdot{\nabla u})&= \lambda{|u|^{{2}^{*}-2}u}+{\mu {|u|^{p-2}u}}& \ \ \mbox{in} \ \ \ {{\mathbb{R}}^{N}_{+}}, \\
		 \frac{{\partial u}}{{\partial n}}&=\sqrt{\lambda}|u|^{{2}_{*}-2}u \ & \mbox{on}\ {{\partial {{\mathbb{R}}^{N}_{+}}}},
	\end{aligned}
	\right.
\end{equation}
where $ \mathbb{R}^{N}_{+}=\{(x{'}, x_{N}): x{'}\in {\mathbb{R}}^{N-1}, x_{N}>0\}$, $N\geq3$, $\lambda>0$, $\mu\in \mathbb{R}$, $2< p <{2}^{*}$, $n$ is the outward normal vector at the boundary ${{\partial {{\mathbb{R}}^{N}_{+}}}}$,
$2^{*}=\frac{2N}{N-2}$ is the usual critical exponent for the Sobolev embedding $D^{1,2}({\mathbb{R}}^{N}_{+})\hookrightarrow {L^{{2}^{*}}}({\mathbb{R}}^{N}_{+})$ and ${2}_{*}=\frac{2(N-1)}{N-2}$ is the critical exponent for the Sobolev trace embedding  $D^{1,2}({\mathbb{R}}^{N}_{+})\hookrightarrow {L^{{2}_{*}}}(\partial \mathbb{R}^{N}_{+})$. By establishing an improved Pohozaev identity, we show that problem (\ref {1.1ab}) has no nontrivial  solution if $\mu \le 0$; By applying the Mountain Pass Theorem without $(PS)$ condition and the delicate estimates for Mountain Pass level, we obtain the existence of a positive solution for all $\lambda>0$ and the different values of the parameters $p$ and ${\mu}>0$. Particularly,  for $\lambda >0$, $N\ge 4$,  $2<p<2^*$, we prove that  problem (\ref {1.1ab}) has a positive solution if and only if $\mu >0$. Moreover, the existence of multiple solutions for  (\ref {1.1ab}) is also obtained by dual variational principle for all $\mu>0$ and suitable $\lambda$.
\end{abstract}

\maketitle
\small
\vs{2mm}
\section{Introduction}
\vs{2mm}
\setcounter{equation}{0}

 In this paper, we concern with the existence and multiplicity of solutions for the following Neumann problem with critical growth
\begin{equation}\label{1.1}
	\left\{
	\begin{aligned}
		-\Delta {u}-\frac{1}{2}(x \cdot{\nabla u})&= \lambda{|u|^{{2}^{*}-2}u}+{\mu {|u|^{p-2}u}}&\ \ \mbox{in} \ \ \ {{\mathbb{R}}^{N}_{+}}, \\
		 \frac{{\partial u}}{{\partial n}}&=\sqrt{\lambda}|u|^{{2}_{*}-2}u \ & \mbox{on}\ {{\partial {{\mathbb{R}}^{N}_{+}}}},
	\end{aligned}
	\right.
\end{equation}
where ${{\mathbb{R}}^{N}_{+} }:=\{(x{'},x_{N}): x'\in {\mathbb{R}}^{N-1},x_{N}>0 \}$ is the upper half-space, $N\geq3$, $\lambda>0$, $\mu\in\mathbb{R}$, $2< p <2^{*}$, $n$ is the outward normal vector at the boundary ${{\partial {{\mathbb{R}}^{N}_{+}}}}$, $2^{*}=\frac{2N}{N-2}$ is the usual critical exponent for the Sobolev embedding $D^{1,2}({\mathbb{R}}^{N}_{+})\hookrightarrow {L^{{2}^{*}}}({\mathbb{R}}^{N}_{+})$ and ${2}_{*}=\frac{2(N-1)}{N-2}$ is the critical exponent for the Sobolev trace embedding  $D^{1,2}({\mathbb{R}}^{N}_{+})\hookrightarrow {L^{{2}_{*}}}(\partial \mathbb{R}^{N}_{+})$. For convenience, we denote $ \mathbb{R}^{N-1}:=\partial\mathbb{R}^{N}_{+}$ and $\int_{\mathbb{R}^{N-1}}:=\int_{\partial\mathbb{R}^{N}_{+}}$.

Our motivation of investigating problem (\ref{1.1}) relies on the fact that, for $p=2$, ${\mu}=\frac{N-2}{4}$, problem (\ref{1.1}) appears when one tries to find the self-similar solutions with special type (see \cite{29, 2})
\begin{equation*}
v(x,t)=t^{-{\mu}}u(\frac{x}{\sqrt{t}}),~~x \in {{\mathbb{R}}_+^{N}},~~t>0,
\end{equation*}
for the nonlinear heat equation
\begin{equation}\label{1.2}
	\left\{
	\begin{aligned}
		v_{t}-\Delta {v}&=\lambda{|v|^{{2}^{*}-2}v}& \ \ \mbox{in} \ \ \ {{\mathbb{R}}^{N}_{+}}{\times(0,\infty)},\\
		\frac{{\partial v}}{{\partial n}}&=\sqrt{\lambda}{|v|^{{2}_{*}-2}v} \ & \mbox{on} \ {{{\mathbb{R}}^{N-1}}}{\times(0,\infty)}.
	\end{aligned}
	\right.
\end{equation}
A simple calculation shows that $v$ is a solution of equation \eqref{1.2} if $u$ solves \eqref{1.1} with $p=2$ and $\mu =\frac {N-2}{4}$. Self-similar solutions or self-similar variables are important because they preserve the PDE scaling and carry simultaneously information about small and large scale behaviors. Self-similar solutions also provide  qualitative properties like global existence, blow-up and asymptotic behavior (see \cite{CG2021,2,3,4}).

In general, consider the nonlinear boundary value problem
\begin{equation}\label{2}
\left\{
	\begin{aligned}
		-\Delta {u}&=f(x,u,{\nabla u})& \ \ \mbox{in} \ \ \ \Omega, \\
		 \frac{{\partial u}}{{\partial n}}&=g(x,u) \ & \mbox{on}\ {\partial\Omega},
	\end{aligned}
	\right.
\end{equation}
where ${\Omega}\subset{{\mathbb{R}}^{N}}$ and $n$ is the outward normal vector on the boundary ${{\partial \Omega}}$.
Equation (\ref{2}) not only has strong research significance in mathematics, but also can be used to describe many physical and biological phenomena, such as in the study of scalar curvature problems and conformal deformation of Riemannian manifolds (see \cite{01, 9}),  problems of sharp constant in Sobolev trace inequalities (see \cite{18}), population genetics (see \cite{03}), non-Newtonian fluid mechanics (see \cite{04}) and so on.

There are several outstanding works when the function $f$ does not depend on ${\nabla u}$. If both $f$ and $g$ have subcritical growth, \eqref{2} has been studied in \cite{5, 113, 114, 21}.
However, if $f$ or $g$ has critical growth (see \cite{22, 23, 24, 26}), proving the existence of  solutions to equation \eqref{2} becomes difficult.
The main difficulty is that Sobolev embedding or Sobolev trace embedding is not compact. As a result, the functional corresponding to equation \eqref{2} does not satisfy the $(PS)$ condition. To overcome this difficulty, one usually uses the $(PS)_{c}$ condition to substitute the $(PS)$ condition, where $c$ is strictly smaller than the energy threshold.
For example, Wang \cite{7} studied equation \eqref{2} when
\begin{equation*}
f(x,u, \nabla u)={u^{{2}^{*}-1}}+h(x,u),
~~~~g(x,u)=-\alpha(x)u ,
\end{equation*}
where ${\Omega}\subset{{\mathbb{R}}^{N}}$ is a bounded domain with $C^{1}$ boundary, $N\geq 3$, $h(x,u)$ is a subcritical perturbation at infinity, $h(x, 0)=0$, and $\alpha(x)$ is a nonnegative function. By using a variant of the Mountain Pass Theorem, Wang obtained the existence of a positive solution with the Mountain Pass value $c\in (0,\frac{1}{2N}S^{\frac{N}{2}})$, where $S$ is the best constant for the Sobolev embedding  $D^{1,2}({\mathbb{R}}^{N})\hookrightarrow {L^{{2}^{*}}}({\mathbb{R}}^{N})$  given by
\begin{equation}\label{1.3}
S:=\inf_{{u\in D^{1,2}({\mathbb{R}}^{N})\setminus \{0\}}}{\frac{{\|\nabla{u}\|_{{L^{2}}(\mathbb{R}^{N})}^{2}}}{{{\|{u}\|_{L^{{2}^{*}}({\mathbb{R}}^{N})}^{2}}}}}.
\end{equation}
In \cite{8}, Deng et al. investigated the existence of a positive solution for equation (\ref{2}) with
\begin{equation}\label{003}
f(x,u, \nabla u)={u^{{2}^{*}-1}}+h(x,u),
~~~~g(u)={u^{{2}_{*}-1}},
\end{equation}
where ${\Omega}$ is a bounded domain in ${{\mathbb{R}}^{N}}$ with $C^{1}$ boundary, $N\geq 3$, $h(x,u)$ is a subcritical perturbation at infinity, and $h(x, 0)=0$.
Due to the fact that both $f$ and $g$ in (\ref{003}) have critical Sobolev growth, Deng et al. no longer used \eqref{1.3} and considered a result of Escobar \cite{9} that the best Sobolev constant $S_{a,b}$ in the  following infimum
\begin{equation*}
S_{a,b}:=\inf_{{u\in D^{1,2}({\mathbb{R}}^{N}_{+})}\setminus \{0\}}\frac{{\|\nabla{u}\|_{{L^{2}}(\mathbb{R}^{N}_{+})}^{2}}}{a {{\|{u}\|_{L^{{2}^{*}}({\mathbb{R}}^{N}_{+})}^{2}}}+b{{\|{u}\|_{L^{{2}_{*}}({\mathbb{R}}^{N-1})}^{2}}}}
\end{equation*}
is achieved by the function ${\varphi(x)}=(1+|x{'}|^{2}+|x_{N}+x_{N}^{0}|^{2})^{\frac{2-N}{2}}$,
where $a, b$ are nonnegative constants with $a+b>0$, $ x_{N}^{0}$ is a constant depending only on $a, b, N$.

The problem (\ref{2}) turns to be more complicated if the function $f$ also depends on $\nabla{u}$.
Based on the research of self-similar solutions for the nonlinear heat equation, many researchers are concerned with  the existence and multiplicity of solutions for the nonlinear boundary value problem
\begin{equation}\label{4}
\left\{
	\begin{aligned}
		-\Delta {u}&={\mu u}+\frac{1}{2}(x \cdot{\nabla u})+a{|u|^{{p}-2}u}& \ \ \mbox{in} \ \ \ {\mathbb{R}}^{N}_{+},  \ \\
		 \frac{{\partial u}}{{\partial n}}&=\gamma {|u|^{{r}-2}u}+{|u|^{{q}-2}u} \ & \mbox{on}\ {\mathbb{R}}^{N-1}.
	\end{aligned}
	\right.
\end{equation}
In \cite{10}, Ferreira et al. investigated equation \eqref{4} with $N\geq3$, $\mu\in\mathbb{R}$, $a\in\{0,1\}$, $2<p<2^{*}$, $\gamma=0$ and $2<q<2_{*}$. By using (Symmetric) Mountain Pass Theorem, they proved the existence of a positive solution and infinitely many solutions for equation \eqref{4}. Recently, Ferreira et al. \cite{11} considered equation \eqref{4} with $N\geq3$, $\mu\in\mathbb{R}$, $a\in \{0, 1\}$, $\gamma =0$, $p=2^{*}$and $q=2_{*}$. They first showed that the best constant $S_{K}$ given by
\begin{equation}\label{6}
S_{K}:=\inf_{{u\in {D_{K}^{1,2}}({\mathbb{R}}^{N}_{+})\setminus \{0\}}} \frac{{\int_{\mathbb{R}^{N}_{+}}K(x)|\nabla{u}|^{2}dx}}
{\Big({\int_{\mathbb{R}^{N-1}}K(x{'},0)|u|^{{2}_{*}}dx{'}}\Big)^{2/2_{*}}}=S_{T},
\end{equation}
where $S_{T}$ is the best constant of the Sobolev trace embedding $D^{1,2}({\mathbb{R}}^{N}_{+})\hookrightarrow {L^{{2}_{*}}}({\mathbb{R}}^{N-1})$ (see \cite{17,18}) defined by
\begin{equation*}
S_{T}:=\inf_{{u\in D^{1,2}({\mathbb{R}}^{N}_{+})}\setminus \{0\}}{\frac{{\|\nabla{u}\|_{{L^{2}}(\mathbb{R}^{N}_{+})}^{2}}}{{{\|{u\|_{L^{{2}_{*}}({\mathbb{R}}^{N-1})}^{2}}}}}},
\end{equation*}
and $K(x):=e^{|x|^{2}/4}$, $D_{K}^{1,2}(\mathbb{R}^{N}_{+})$ is the closure of ${{C_{c}^{\infty}}(\overline{{{\mathbb{R}}^{N}_{+}}})}$ with respect to the following norm
\begin{equation*}
 \|{u}\|:= \Big({\int_{\mathbb{R}^{N}_{+}}K(x)|\nabla{u}|^{2}dx} \Big)^{\frac{1}{2}}.
\end{equation*}
Next, they established the existence of a positive solution for (\ref {4}) with $\gamma =0$, $p=2^{*}$ and $q=2_{*}$
 if either $a=0$, $N\geq7$ and $\mu\in(\frac{N}{4}+\frac{(N-4)}{8}, \frac{N}{2})$ or $a=1$, $N\geq 3 $ and $\mu\in(\frac{N}{2}-\delta,  \frac{N}{2})$, where $\delta>0$ is a small constant.
   Moreover, some interesting nonexistence results were obtained for problem (\ref {4}) with $N\geq 3$, $\gamma =0$, $p=2^{*}$,  $q=2_{*}$ if $\mu \in (-\infty, \frac N4)\cup [\frac N2, +\infty)$.
 Particularly, a nonexistence result of self-similar solutions to problem \eqref{1.2} with $\lambda=1$ was derived. For $\gamma>0$, Furtado and da Silva \cite{16} obtained the existence of a positive solution for \eqref{4} by the infimum \eqref{6} when $N\geq 4$, $\mu=0$, $a=0$, $2\leq r<2_{*}$ and $q=2_{*}$.
We also refer the interested readers to \cite{013, 021, 022} and their references for various results.

Inspired mainly by  \cite{11, 10, 16}, we consider the existence and multiplicity of solutions for problem \eqref{1.1}. In view of  $\nabla{K(x)}=\frac{1}{2}xK(x)$,  by directly calculating, equation \eqref{1.1} can be written as
\begin{equation}\label{1.5}
	\left\{
	\begin{split}
		-\mbox{div}(K(x)\nabla {u})&= {\mu K(x) {|u|^{p-2}u}}+\lambda K(x){|u|^{{2}^{*}-2}u} \ \ \hspace{0.5cm} \hbox{in} \ \ {\mathbb{R}}^{N}_{+}, \ \  \\
		K(x{'},0)\frac{{\partial u}}{{\partial n}}&=\sqrt{\lambda}K(x{'},0){|u|^{{2}_{*}-2}u} \ \  \ \hspace{2.05cm} \mbox{on} \ {\mathbb{R}}^{N-1},
	\end{split}
	\right.
\end{equation}
which implies that we only need to study the existence and multiplicity of solutions for (\ref {1.5}). It is natural to look for solutions of \eqref{1.5} in the weighted Sobolev space ${{D_{K}^{1,2}}({\mathbb{R}}^{N}_{+})}$. For simplicity, we denote ${{D_{K}^{1,2}}({\mathbb{R}}^{N}_{+})}$ by $X$. 
For any  $2\leq r\leq 2^{*}, 2\leq q\leq 2_{*}$, define the weighted Lebesgue spaces
\begin{eqnarray}
&&{\ \ \ L^{r}_{K}({\mathbb{R}}^{N}_{+})}:= \bigg\{u \in {L^{r}({\mathbb{R}}^{N}_{+})}:{\int_{\mathbb{R}^{N}_{+}}K(x)|u|^{r}dx} < {\infty}\bigg\}, \notag\\
&&{L^{q}_{K}({\mathbb{R}}^{N-1})}:= \bigg\{u \in {L^{q}({\mathbb{R}}^{N-1})}:{\int_{\mathbb{R}^{N-1}}K(x{'},0)|u|^{q}dx{'}} < {\infty}\bigg\}. \notag
\end{eqnarray}
In \cite{10, 11}, Ferreira et al. proved that the embedding $X \hookrightarrow L^{r}_{K}(\mathbb{R}^{N}_{+})$ is continuous for $r\in[2, 2^{*}]$ and compact for $r\in[2, 2^{*})$, the embedding $X\hookrightarrow {L^{q}_{K}}({\mathbb{R}}^{N-1})$ is continuous for $q\in[2, 2_{*}]$ and compact for $q\in[2, 2_{*})$. Moreover, the first eigenfunction $\varphi_1$ of the linear problem
\begin{equation*}
\left\{
	\begin{aligned}
		-\Delta {u}-\frac{1}{2}(x \cdot{\nabla u})&={\hat{\lambda} u}& \ \ \mbox{in} \ \ \ {\mathbb{R}}^{N}_{+},  \ \\
		 \frac{{\partial u}}{{\partial n}}&=0& \mbox{on}\ {\mathbb{R}}^{N-1},
	\end{aligned}
	\right.
\end{equation*}
is positive or negative. Here we assume that $\varphi_1$ is a positive function. And the corresponding first eigenvalue is characterized by
\begin{equation}
\hat{\lambda}_{1}:=\inf_{u\in X\setminus \{0\}} \frac{ \|{u}\|^{2}}
{{{\|{u}\|}}_{{L_K^{2}({\mathbb{R}}^{N}_{+})}}^{2}}=\frac{N}{2}.\\\label{1.6}
\end{equation}

The energy functional $J_{\lambda, \mu}: X\rightarrow \mathbb{R}$ associated to \eqref{1.5} is defined  by
\begin{equation*}
	J_{\lambda, \mu}(u):={\frac{1}{2}}{\|{u}\|}^{2}-{\frac{\mu }{p}}{\|u\|}_{{L_{K}^{p}}({\mathbb{R}}^{N}_{+})}^{p}-{\frac{\lambda }{2^{*}}}\|u\|_{L^{{2}^{*}}_{K}({\mathbb{R}}^{N}_{+})}^{2^{*}}-{\frac{\sqrt{\lambda} }{2_{*}}}\|u\|_{L^{{2}_{*}}_{K}({\mathbb{R}}^{N-1})}^{2_{*}}.
\end{equation*}
The embedding results in \cite{10, 11} show that $J_{\lambda, \mu}$ is well defined and belongs to $C^{1}(X, \mathbb{R})$. Therefore, for any $u, v\in X$, we have
\begin{equation*}
	\begin{aligned}
		\langle J'_{\lambda, \mu}(u), v \rangle
		&=\int_{\mathbb{R}^{N}_{+}}\big( K(x){\nabla{u}}{\nabla{v}}-\mu K(x)|u|^{p-2}uv-\lambda K(x)|u|^{2^{*}-2}uv\big)dx\\
		&  \ \ \  \ -{\int_{\mathbb{R}^{N-1}}\sqrt{\lambda}K(x{'},0)|u|^{2_{*}-2}uvdx{'}}.
	\end{aligned}
\end{equation*}
Introduce now the modified functional
\begin{equation}
	I_{\lambda, \mu}(u):={\frac{1}{2}}{\|{u}\|}^{2}-{\frac{\mu }{p}}{\|{u_{+}}\|}_{{L_{K}^{p}}({\mathbb{R}}^{N}_{+})}^{p}-{\frac{\lambda }{2^{*}}}{\|{u_{+}}\|}_{L^{{2}^{*}}_{K}({\mathbb{R}}^{N}_{+})}^{2^{*}}-{\frac{\sqrt{\lambda}}{2_{*}}}{\|{u_{+}}\|}_{L^{{2}_{*}}_{K}({\mathbb{R}}^{N-1})}^{2_{*}}  ,\notag
\end{equation}
where $u_{+}= \max \{0,u\}$, $u_{-}= - \min \{0,u\}$. Obviously, any weak solution of \eqref{1.5} is a critical point of $J_{\lambda, \mu}$. The functional $I_{\lambda, \mu}$ is introduced in order to obtain nonnegative critical points for $J_{\lambda, \mu}$. Indeed, if $u\in X$ is a nonzero critical point of $I_{\lambda, \mu}(u)$, then
\begin{equation}
	0=\langle {I{'}_{\lambda, \mu}(u)},{u_{-}}\rangle ={\|{u_{-}}\|}^{2}.\notag
\end{equation}
It follows from (\ref{1.6}) that $u_{-}\equiv 0$. Hence,  it suffices to find a nonzero critical point of $I_{\lambda, \mu}$ in order to obtain a nonnegative weak solution of \eqref{1.5}.

The aim of this paper is to establish the existence of solutions according to the natural range for the parameters $\lambda$, $\mu$ and $p$. Firstly, using an improved Pohozaev identity and a Hardy-type inequality, we obtain a nonexistence result.

\begin{Thm}\label{Th1.0}
	Let $N\geq3, \lambda>0, \mu\leq0$ and $2<p<2^*$. Suppose that $u\in C^2(\mathbb{R}^N_+)\cap X$ is a solution of equation \eqref{1.1}, then $u\equiv0$.
\end{Thm}
Next, if $\mu >0$, we give the existence Theorem for problem (\ref {1.1}) as follows:
\begin{Thm}\label{Th1.1}
	For any fixed $\lambda >0$, equation \eqref{1.1} has a positive solution if  one of the following three assumptions holds:
	\vskip 0.2cm
	
	
	$(\romannumeral 1)$  $N\geq4$, $2<p<2^*$ and $\mu>0$;
\vskip 0.2cm
	
	$(\romannumeral 2)$  $N=3$, $4<p<6$ and $\mu>0$;
	\vskip 0.2cm
	
	$(\romannumeral 3)$   $N=3$, $2<p\leq4$ and $ \mu>0$ sufficiently large.
\end{Thm}

\begin{Rem}
It follows from Theorem \ref {Th1.0} and Theorem \ref {Th1.1} that, for any fixed $\lambda >0$, $N\ge 4$,  $2<p<2^*$,  the problem  (\ref {1.1}) has a positive solution if and only if $\mu >0$. The case for $N=3$, which correspond to the critical dimension, is very complicate. We only get an existence result for $4<p<6$, $\mu >0$ or $2<p\le 4$, $\mu >0$ large enough.
\end{Rem}

\smallskip
In the proof of Theorem \ref{Th1.1}, we mainly apply the ideas introduced by Brezis and Nirenberg in \cite{0002}.
The difficulties here lie in two aspects. The first difficulty is the lack of compactness for the embedding 
$X\hookrightarrow {L_{K}^{2^{*}}}({\mathbb{R}}^{N}_+)$ and 
$X\hookrightarrow {L_{K}^{2_{*}}}({\mathbb{R}}^{N-1})$, which causes the functional $I_{\lambda, \mu}$ not satisfying the $(PS)$ condition.
To overcome this difficulty, we look for a threshold value of functional under which the $(PS)$ sequence is pre-compact, and this idea was originally proposed in \cite{0002, 7}.
The second difficulty is the selection of an appropriate test function in the new phenomenon.
Here, different from the test function used in \cite{8}, we use the following test function
\begin{equation*}
\tilde{U}_{{\varepsilon}}(x)
=K(x)^{-\frac{1}{2}}\phi(x)\frac{\big({\varepsilon^2} N(N-2)\big) ^{\frac{N-2}{4}}}{\big(\varepsilon^{2}+|x'|^{2}+|x_{N}+\varepsilon x_{N}^{0} |^{2}\big)^{\frac{N-2}{2}}},
\end{equation*}
where $\phi(x) \in C_{0}^{\infty}({\mathbb{R}}^{N},[0,1])$ is a cut-off function, $\varepsilon>0$ and $x_{N}^{0} =(N/(N-2))^{1/2}$. Moreover, we perform some fine estimates concerning the asymptotic behavior of $\tilde{U}_{{\varepsilon}}$ when $\varepsilon$ is tending to $0$.
\smallskip

Finally, it is natural to use the dual variational principle to consider the multiplicity of solutions since the functional $J_{\lambda, \mu}$ is even.

\begin{Thm}\label{Th1.2}
	If $N\geq 3$, $p\in(2,2^{*})$ and $\mu>0$, then for each $k=1, 2, \cdots,$ there exists a sequence $\{\lambda_{k}\}\subset (0, +\infty)$ such that equation \eqref{1.1} has $k$ pair of solutions $\{u_{j},-u_{j}\}$, $j=1, 2, \cdots, k$, provided $\lambda\in (0,\lambda_k)$.
\end{Thm}

The paper is organized as follows. By means of an improved Pohozaev identity and a Hardy-type inequality, a nonexistence result is obtained  in Section \ref{S1}. In Section \ref{S2}, we verify that $I_{\lambda, \mu}$ satisfies the geometric conditions of the Mountain Pass Theorem and establish the local compactness for $I_{\lambda, \mu}$ under the assumption
\eqref{2.8}. In Section \ref{S3}, we complete the proof of Theorem \ref {Th1.1} by verifying that the assumption \eqref{2.8} holds. We are committed to providing careful estimates of $\tilde{U}_{\varepsilon}$. 
In Section \ref{S4}, we consider the existence of multiple solutions by dual variational principle and finish the proof of Theorem \ref {Th1.2}.

\vs{3mm}
\section{The Nonexistence result}\label{S1}
\vs{2mm}

In this section, we establish a nonexistence result for problem \eqref{1.1}. To this end, we consider the following general Neumann problem
\begin{equation}\label{a1}
	\left\{
	\begin{aligned}
		-\Delta {u}-\frac{1}{2}(x \cdot{\nabla u})&= f(u)&\ \ \mbox{in} \ \ \ {{\mathbb{R}}^{N}_{+}}, \\
		 \frac{{\partial u}}{{\partial n}}&=g(u) \ & \mbox{on}\ { {{\mathbb{R}}^{N-1}}},
	\end{aligned}
	\right.
\end{equation}
where $N\geq3$ and functions $f$, $g$ satisfy the assumptions stated below:

$(f_1)$  $f, g:\mathbb{R}\rightarrow \mathbb{R}$ are continuous;

$(f_2)$  there exist two positive constants $C_1$ and $C_2$ such that
$$
0\le tf(t) \le C_1(t^2+|t|^{2^*}) \ \ \mbox{and} \ \ 0\le tg(t) \le C_2(t^2+|t|^{2_*})
$$
for all $t\in \mathbb{R}$.
	
We first state an improved Pohozaev identity for problem \eqref{a1} by a truncation argument.

\begin{Lem}\label{lem1.1} (Pohozaev identity) Suppose that $N\geq3$ and $f$, $g$ satisfy assumptions $(f_1)$-$(f_2)$. If $u\in C^2(\mathbb{R}^N_+)\cap X$ is a solution of problem \eqref{a1}, then there hold
\begin{equation}\label{a2}
{\|\nabla{u}\|_{{L^{2}}(\mathbb{R}^{N}_{+})}^{2}}-{\int_{\mathbb{R}^{N}_{+}}uf(u)dx}
-{\int_{\mathbb{R}^{N-1}}ug(u)dx'}
=-\frac{N}{4}\|u\|_{{L^{2}}(\mathbb{R}^{N}_{+})}^{2}
\end{equation}
and
\begin{equation}\label{a3}
\frac{N-2}{2}{\|\nabla{u}\|_{{L^{2}}(\mathbb{R}^{N}_{+})}^{2}}
-N{\int_{\mathbb{R}^{N}_{+}}F(u)dx}
-(N-1){\int_{\mathbb{R}^{N-1}}G( u)dx'}=-\frac{1}{2}{\int_{\mathbb{R}^{N}_{+}}(x \cdot{\nabla u})^2dx}.
\end{equation}
\end{Lem}

\begin{proof}
Let $\psi \in C_0^{\infty}([0,\infty),[0,1])$ be a cut-off function such that  $\psi \equiv 1$ in $[0, 1]$, $\psi\equiv 0$ in $[4, \infty)$
and $\|\psi'\|_{{L^{\infty}}([0,\infty))}$ is bounded. For any $k\geq1$, define $\psi_k(x):=\psi({|x|^2}/{k^2})$.

Firstly, multiplying the first equation of \eqref{a1} by $\psi_ku$ and integrating both sides over $\mathbb{R}^{N}_{+}$, we have
\begin{equation}\label{a4}
		-\int_{\mathbb{R}^{N}_{+}}\psi_ku\Delta {u}dx-\frac{1}{2}\int_{\mathbb{R}^{N}_{+}}\psi_ku(x \cdot{\nabla u})dx=\int_{\mathbb{R}^{N}_{+}}\psi_ku f(u)dx.
\end{equation}
From the divergence Theorem, we conclude that
\begin{equation*}
	\begin{aligned}
		-\int_{\mathbb{R}^{N}_{+}}(\psi_ku)\Delta {u}dx
&=\int_{\mathbb{R}^{N}_{+}}u(\nabla{\psi_k}\cdot\nabla {u})dx+
\int_{\mathbb{R}^{N}_{+}}{\psi_k}|\nabla {u}|^2dx-
\int_{\mathbb{R}^{N-1}}\psi_k  u\frac{{\partial u}}{{\partial n}}dx'\\
&=\frac{2}{k^2}\int_{\mathbb{R}^{N}_{+}}{\psi'}\Big(\frac{|x|^2}{k^2}\Big)(x\cdot\nabla {u})udx+
\int_{\mathbb{R}^{N}_{+}}{\psi_k}|\nabla {u}|^2dx-
\int_{\mathbb{R}^{N-1}}\psi_k ug(u)dx'.
	\end{aligned}
\end{equation*}
Since $u\in X$ and $\|\psi'\|_{{L^{\infty}}([0,\infty))}$ is bounded, we get for $k$ large enough,
\begin{equation}\label{a4'}
		-\int_{\mathbb{R}^{N}_{+}}(\psi_ku)\Delta {u}dx
=\int_{\mathbb{R}^{N}_{+}}{\psi_k}|\nabla {u}|^2dx-
\int_{\mathbb{R}^{N-1}}\psi_k ug(u)dx'+o_k(1),
\end{equation}
where $o_k(1)\rightarrow 0$ as $k\rightarrow \infty$. It follows from the Fubini Theorem, the divergence Theorem and the fundamental Theorem of calculus that
\begin{equation}\label{a5}
	\begin{aligned}
		& \ \ \ \ \int_{\mathbb{R}^{N}_{+}}\mbox{div}(\psi_k|u|^2x)dx
		=\int_{B^+_{2k}}\mbox{div}(\psi_k|u|^2x)dx\\
		&=\int_{0}^{2k}\int_{\hat{B}_{2k}} \mbox{div}_{x'}(\psi_k|u|^2x')dx'dx_N+
\int_{\hat{B}_{2k}}\int_{0}^{2k}(\psi_k|u|^2x_N)_{x_N}dx_Ndx'=0,
	\end{aligned}
\end{equation}
where $B_{2k}^{+}:=B_{2k}(0)\cap {\mathbb{R}}^{N}_{+}$ and  $\hat{B}_{2k}:=\hat{B}_{2k}(0)\subset {\mathbb{R}}^{N-1}$. On the other hand, for $k$ large enough, there holds
\begin{equation}\label{a5'}
	\begin{aligned}
		\int_{\mathbb{R}^{N}_{+}}\mbox{div}(\psi_k|u|^2x)dx
&=N\int_{\mathbb{R}^{N}_{+}}\psi_k|u|^2dx
+\int_{\mathbb{R}^{N}_{+}}|u|^2(x\cdot\nabla \psi_k)dx
+2\int_{\mathbb{R}^{N}_{+}}\psi_ku(x\cdot\nabla u)dx\\
&=N\int_{\mathbb{R}^{N}_{+}}\psi_k|u|^2dx
+2\int_{\mathbb{R}^{N}_{+}}\psi_ku(x\cdot\nabla u)dx+o_k(1).
	\end{aligned}
\end{equation}
 Combining \eqref{a5} with \eqref{a5'}, we conclude that 
\begin{equation}\label{a6}
\int_{\mathbb{R}^{N}_{+}}\psi_ku(x\cdot\nabla u)dx=-\frac{N}{2}\int_{\mathbb{R}^{N}_{+}}\psi_k|u|^2dx+o_k(1).
\end{equation}
In view of \eqref{a4}, \eqref{a4'} and \eqref{a6}, one has for $k$ large enough,
\begin{equation*}
\int_{\mathbb{R}^{N}_{+}}{\psi_k}|\nabla {u}|^2dx-\int_{\mathbb{R}^{N}_{+}}\psi_ku f(u)dx-
\int_{\mathbb{R}^{N-1}}\psi_kug(u)dx'
+\frac{N}{4}\int_{\mathbb{R}^{N}_{+}}\psi_k|u|^2dx
=o_k(1).
\end{equation*}
Letting $k\rightarrow\infty$,  \eqref{a2} follows from the  growth condition $(f_2)$, Sobolev embedding and the Lebesgue dominated convergence Theorem.

Next, multiplying the first equation of \eqref{a1} by $\psi_k(x\cdot \nabla u)$ and integrating both sides over $\mathbb{R}^{N}_{+}$, one has
\begin{equation}\label{bb1}
		-\int_{\mathbb{R}^{N}_{+}}\psi_k(x\cdot \nabla u)\Delta {u}dx-\frac{1}{2}\int_{\mathbb{R}^{N}_{+}}\psi_k(x \cdot{\nabla u})^2dx=\int_{\mathbb{R}^{N}_{+}}\psi_k(x\cdot \nabla u)f(u)dx.
\end{equation}
Let $F_1:=(x\cdot \nabla u)\nabla u$ and $F_2:={x|\nabla u|^2}/{2}$. Simple computation yields that
\begin{equation}\label{bb2}
\psi_k(x\cdot \nabla u)\Delta {u}=\psi_k \mbox{div}(F_1-F_2)+\frac{N-2}{2}\psi_k|\nabla u|^2.
\end{equation}
Note that $$\psi_k \mbox{div}(F_1-F_2)
=\mbox{div}\big(\psi_k (F_1-F_2)\big)-(F_1-F_2)\nabla\psi_k.$$
 Using the boundedness of $\psi'$, we get that for $k$ large enough,
\begin{equation*}
		\int_{\mathbb{R}^{N}_{+}}\psi_k \mbox{div}(F_1-F_2)dx
=\int_{\mathbb{R}^{N}_{+}}\mbox{div}\big(\psi_k (F_1-F_2)\big)dx+o_k(1).
\end{equation*}
Similar as in \eqref{a5}, we have
\begin{equation*}
	\begin{aligned}
		& \ \ \ \ \int_{\mathbb{R}^{N}_{+}}\mbox{div}(\psi_kF_1)dx
		=\int_{B^+_{2k}}\mbox{div}\big(\psi_k(x\cdot \nabla u)\nabla u\big)dx\\
		&=\int_{0}^{2k}\int_{\hat{B}_{2k}}\mbox{div}_{x'}\big(\psi_k(x\cdot \nabla u)\nabla_{x'}u\big)dx'dx_N+
		\int_{\hat{B}_{2k}}\int_{0}^{2k}\big(\psi_k(x\cdot \nabla u)u_{x_N}\big)_{x_N}dx_Ndx'\\
		&=-\int_{\hat{B}_{2k}}\psi_k(x'\cdot \nabla_{x'} u)u_{x_N}dx'
		=-\int_{{\mathbb{R}}^{N-1}}\psi_k(x'\cdot \nabla_{x'} u)u_{x_N}dx'\\
		&=\int_{{\mathbb{R}}^{N-1}}\psi_k(x',0)(x'\cdot \nabla _{x'}u)g(u)dx'
	\end{aligned}
\end{equation*}
and 
\begin{equation*}
\int_{\mathbb{R}^{N}_{+}}\mbox{div}(\psi_kF_2)dx=0.
\end{equation*}
Then, we have for $k$ large enough,
\begin{equation}\label{bb3}
		\int_{\mathbb{R}^{N}_{+}}\psi_k \mbox{div}(F_1-F_2)dx
=\int_{{\mathbb{R}}^{N-1}}\psi_k(x'\cdot \nabla_{x'}u)g(u)dx'+o_k(1).
\end{equation}
It follows from \eqref{bb1}-\eqref{bb3} that for $k$ large enough,
\begin{equation}\label{bb4}
	\begin{aligned}
& \ \ \ \ \frac{N-2}{2}\int_{\mathbb{R}^{N}_{+}}\psi_k|\nabla u|^2dx+\int_{\mathbb{R}^{N}_{+}}\psi_k(x\cdot \nabla u)f(u)dx+\int_{{\mathbb{R}}^{N-1}}\psi_k(x'\cdot \nabla_{x'}u)g(u)dx'\\
&=-\frac{1}{2}\int_{\mathbb{R}^{N}_{+}}\psi_k(x \cdot{\nabla u})^2dx+o_k(1).
	\end{aligned}
\end{equation}
Using the same argument as \eqref{a5} gives that
\begin{equation}\label{bb5}
		\int_{\mathbb{R}^{N}_{+}}\mbox{div}(\psi_kF(u)x)dx
=0.
\end{equation}
Moreover, for $k$ large enough,
\begin{equation}\label{bb6}
	\begin{aligned}
		& \ \ \ \ \int_{\mathbb{R}^{N}_{+}}\mbox{div}(\psi_kF(u)x)dx\\
&=N\int_{\mathbb{R}^{N}_{+}}\psi_kF(u)dx
+\int_{\mathbb{R}^{N}_{+}}(x\cdot\nabla \psi_k)F(u)dx
+\int_{\mathbb{R}^{N}_{+}}\psi_kf(u)(x\cdot\nabla u)dx\\
&=N\int_{\mathbb{R}^{N}_{+}}\psi_kF(u)dx
+\int_{\mathbb{R}^{N}_{+}}\psi_kf(u)(x\cdot\nabla u)dx+o_k(1),
	\end{aligned}
\end{equation}
since the condition $(f_2)$ holds. In view of \eqref{bb5} and \eqref{bb6}, one has 
    \begin{equation}\label{bb7}
\int_{\mathbb{R}^{N}_{+}}\psi_kf(u)(x\cdot\nabla u)dx=-N\int_{\mathbb{R}^{N}_{+}}\psi_kF(u)dx+o_k(1).
\end{equation}
Similarly, we obtain
    \begin{equation}\label{bb8}
\int_{\mathbb{R}^{N-1}}\psi_kg(u)(x'\cdot\nabla_{x'}u)dx' =-(N-1)\int_{\mathbb{R}^{N-1}}\psi_kG(u)dx'+o_k(1).
\end{equation}
We conclude from \eqref{bb4}, \eqref{bb7} and \eqref{bb8} that
\begin{equation*}
	\begin{aligned}
& \ \ \ \ \frac{N-2}{2}\int_{\mathbb{R}^{N}_{+}}\psi_k|\nabla u|^2dx-N\int_{\mathbb{R}^{N}_{+}}\psi_kF(u)dx-(N-1)\int_{\mathbb{R}^{N-1}}\psi_kG( u)dx'\\
&=-\frac{1}{2}\int_{\mathbb{R}^{N}_{+}}\psi_k(x \cdot{\nabla u})^2dx+o_k(1).
	\end{aligned}
\end{equation*}
By letting $k\rightarrow\infty$ and using the Lebesgue dominated convergence Theorem, we easily obtain identity \eqref{a3}. 
\end{proof}
Next, we state a Hardy-type inequality which will be necessary to the proof of Theorem \ref{Th1.0}.
\begin{Lem}\label{lem1.2}(\cite{11}, Proposition 3.3)
	If $N\geq3$, then for any $u\in X$, there holds
\begin{equation*}
\frac{N^2}{4}\int_{\mathbb{R}^{N}_{+}} u^2dx\leq\int_{\mathbb{R}^{N}_{+}} (x \cdot{\nabla u})^2dx.
\end{equation*}
\end{Lem}
\begin{proof}[\bf{Proof of Theorem \ref{Th1.0}}]
Taking $f(u):=\lambda{|u|^{{2}^{*}-2}u}+{\mu {|u|^{p-2}u}}$ and  $g(u):=\sqrt{\lambda}{|u|^{{2}_{*}-2}u}$, then we conclude from Lemma \ref{lem1.1} that
\begin{equation*}
{\|\nabla{u}\|_{{L^{2}}(\mathbb{R}^{N}_{+})}^{2}}-\lambda{\|{u}\|}_{L^{{2}^{*}}({\mathbb{R}}^{N}_{+})}^{2^{*}}
-\sqrt{\lambda}{\|{u}\|}_{L^{{2}_{*}}({\mathbb{R}}^{N-1})}^{2_{*}}
-\mu{\|{u}\|}_{L^p({\mathbb{R}}^{N}_{+})}^{p}
=-\frac{N}{4}\|u\|_{{L^{2}}(\mathbb{R}^{N}_{+})}^{2}
\end{equation*}
and
\begin{equation*}
	\begin{aligned}
\frac{N-2}{2}\big({\|\nabla{u}\|_{{L^{2}}(\mathbb{R}^{N}_{+})}^{2}}
-\lambda{\|{u}\|}_{L^{{2}^{*}}({\mathbb{R}}^{N}_{+})}^{2^{*}}
-\sqrt{\lambda}{\|{u}\|}_{L^{{2}_{*}}({\mathbb{R}}^{N-1})}^{2_{*}}\big)
-\frac{\mu N}{p}{\|{u}\|}_{L^p({\mathbb{R}}^{N}_{+})}^{p} =-\frac{1}{2}{\int_{\mathbb{R}^{N}_{+}}(x \cdot{\nabla u})^2dx},
	\end{aligned}
\end{equation*}
which give that
\begin{equation}\label{p1}
\mu \Big(\frac{N}{p}-\frac{N-2}{2}\Big){\|{u}\|}_{L^p({\mathbb{R}}^{N}_{+})}^{p}
=\frac{1}{2}{\int_{\mathbb{R}^{N}_{+}}(x \cdot{\nabla u})^2dx}-\frac{N(N-2)}{8}\|u\|_{{L^{2}}(\mathbb{R}^{N}_{+})}^{2}.
\end{equation}
We derive from (\ref{p1}) and Lemma \ref{lem1.2} that
\begin{equation*}
\frac{N}{4}\|u\|_{{L^{2}}(\mathbb{R}^{N}_{+})}^{2}\leq\mu \Big(\frac{N}{p}-\frac{N-2}{2}\Big){\|{u}\|}_{L^p({\mathbb{R}}^{N}_{+})}^{p}.
\end{equation*}
Due to $2<p<2^*$, one has that $u\equiv0$ if $\mu\leq0$. The proof is finished.
\end{proof}

\vs{2mm}
\section{A local compactness result}\label{S2}
\vs{2mm}

In this section, we are going to verify that $I_{\lambda, \mu}$ satisfies the geometric
conditions of the Mountain Pass Theorem and then establish the local compactness for $I_{\lambda, \mu}$ under the assumption \eqref{2.8}. In the following, we always assume that $N\geq3$, $\lambda, \mu>0$ and $2< p<2^*$.

From Theorem 3.3 in \cite{01}, we can conclude the following Lemma.
\begin{Lem}\label{lem2.1}
	For any $\theta \in (0,1]$, the infimum
\begin{equation}
S_{\theta}:=\inf_{{u\in D^{1,2}({\mathbb{R}}^{N}_{+})}\setminus \{0\}}\frac{{\|\nabla{u}\|_{{L^{2}}(\mathbb{R}^{N}_{+})}^{2}}}{\theta {{\|{u}\|_{L^{{2}^{*}}({\mathbb{R}}^{N}_{+})}^{2}}}+(1-\theta){{\|{u}\|_{L^{{2}_{*}}({\mathbb{R}}^{N-1})}^{2}}}}\\\label{2.1}
\end{equation}
is achieved by
 the function
\begin{equation}
{\varphi_{\varepsilon}(x)}=\bigg(\frac{\varepsilon}{{\varepsilon}^{2}+|x{'}|^{2}+|x_{N}+{\varepsilon}x_{N}^{0}|^{2}}\bigg)^{\frac{N-2}{2}},\notag
\end{equation}
where $\varepsilon>0$, $x'\in \mathbb{R}^{N-1}$, $ x_{N}^{0}$ is a constant depending only on $\theta$ and $N$.
\end{Lem}

For $\tau\geq0$, set
\begin{equation}\label{2.2}
{\varphi}_{{\varepsilon},{\tau}}(x):=\bigg(\frac{{\varepsilon}{\sqrt{N(N-2)}}}{{\varepsilon}^{2}+|x{'}|^{2}+|x_{N}+{\varepsilon}{\tau}x_{N}^{0}|^{2}}\bigg)^{\frac{N-2}{2}}, \
x_{N}^{0}:={\sqrt{\frac{N}{N-2}}}.
\end{equation}
It is easy to check that ${\varphi}_{{\varepsilon},{\tau}}$ satisfies
\begin{equation}\label{2.3}
	\left\{
	\begin{aligned}
		-\Delta {u}&=u^{{2}^{*}-1}& \ \ \mbox{in} \ \ \ {{\mathbb{R}}^{N}_{+}}, \ \ \\
		\frac{{\partial u}}{{\partial n}}&=\tau u^{{2}^{*}-1} \ & \ \ \mbox{on}\ \mathbb{R}^{N-1}.
	\end{aligned}
	\right.
\end{equation}
Let
\begin{equation*}
 {\theta}:=\frac{{{\|{{\varphi}_{{\varepsilon},{\tau}}}\|_{L^{{2}^{*}}({\mathbb{R}}^{N}_{+})}^{2^{*}-2}}}}
{{{\|{{\varphi}_{{\varepsilon},{\tau}}}\|_{L^{{2}^{*}}({\mathbb{R}}^{N}_{+})}^{2^{*}-2}}}+{\tau}{{\|{{{\varphi}_{{\varepsilon},{\tau}}}}\|_{L^{{2}_{*}}({\mathbb{R}}^{N-1})}^{2_{*}-2}}}}, \end{equation*}
which is independent of $\varepsilon$. Then ${\varphi}_{{\varepsilon},{\tau}}(x)$ reaches the infimum $S_{\theta}$.

Denote
\begin{equation*}
\Phi_{\lambda}(u):={\frac{1}{2}}{\|\nabla{u}\|_{{L^{2}}(\mathbb{R}^{N}_{+})}^{2}}-{\frac{\lambda }{2^{*}}}{\|{u_{+}}\|}_{L^{{2}^{*}}({\mathbb{R}}^{N}_{+})}^{2^{*}}-{\frac{\sqrt{\lambda} }{2_{*}}}{\|{u_{+}}\|}_{L^{{2}_{*}}({\mathbb{R}}^{N-1})}^{2_{*}} 
\end{equation*}
and set
\begin{equation}
A_\lambda:={\inf\limits_{{u\in D^{1,2}({\mathbb{R}}^{N}_{+})}\setminus\{0\}}}\ {\sup\limits_{t>0}}\ \Phi_{\lambda}(tu).\label{2.5}
\end{equation}
\begin{Lem}\label{lem2.2}(\cite{31}, Lemma 2.4)
	The infimum $A_{\lambda}$ is achieved by
$\psi_{\lambda,\varepsilon}=\lambda^{-\frac{N-2}{4}}{\varphi}_{{\varepsilon},{1}}$
and
\begin{equation*}
A_\lambda={\lambda^{-\frac{N-2}{2}}}\bigg({\frac{1}{2}}\|\nabla{{\varphi}_{{\varepsilon},{1}}}\|_{{L^{2}}(\mathbb{R}^{N}_{+})}^{2}-
{\frac{1}{2^{*}}}{\|{\varphi}_{{\varepsilon},{1}}\|}_{L^{{2}^{*}}({\mathbb{R}}^{N}_{+})}^{2^{*}}
-{\frac{1} {2_{*}}}{\|{\varphi}_{{\varepsilon},{1}}\|}_{L^{{2}_{*}}({\mathbb{R}}^{N-1})}^{2_{*}}  \bigg).
\end{equation*}
\end{Lem}

For simplicity, we define
\begin{equation}\label{2.111}
{U}_{\varepsilon}:={\varphi}_{{\varepsilon},{1}},~
K_{1}:={\|\nabla{{U}_{\varepsilon}}\|_{{L^{2}}(\mathbb{R}^{N}_{+})}^{2}},~
K_{2}:={\|{{U_{\varepsilon}}}\|}_{L^{{2}^{*}}({\mathbb{R}}^{N}_{+})}^{2^{*}},~
K_{3}:={\|{{U_{\varepsilon}}}\|}_{L^{{2}_{*}}({\mathbb{R}}^{N-1})}^{2_{*}}.
\end{equation}
It follows from Lemma \ref{lem2.2} and \eqref{2.3} that
\begin{equation}\label{2000}
A_\lambda=\lambda^{-\frac{N-2}{2}}A,
~A={\frac{K_{1}}{2}}-{\frac{K_{2} }{2^{*}}}-{\frac{K_{3} }{2_{*}}} \ \ and \ \
~K_{1}-K_{2}-K_{3}=0.
\end{equation}

The proof of following Lemma is similar to Theorem 1.1 in \cite{11}, we omit details  here.
\begin{Lem}\label{lem2.4}
	For any $\theta \in (0,1]$, let
\begin{equation}
S_{\theta}^{K}:=\inf_{u\in X\setminus \{0\}}\frac{{\|\nabla{u}\|_{{L_{K}^{2}}(\mathbb{R}^{N}_{+})}^{2}}}{\theta {{\|{u}\|_{L^{{2}^{*}}_{K}({\mathbb{R}}^{N}_{+})}^{2}}}+(1-\theta){{\|{u}\|_{L^{{2}_{*}}_{K}({\mathbb{R}}^{N-1})}^{2}}}},\notag
\end{equation}
then $S_{\theta}^{K}=S_{\theta}$.
\end{Lem}

 Now, we are going to verify that $I_{\lambda, \mu}$ has a Mountain Pass structure.

\begin{Lem}\label{lem2.3}
The functional $I_{\lambda, \mu}$ satisfies the following three items:

$(I_1)~I_{\lambda, \mu}(0)=0;$

$(I_2)$~\text{there exist}~${\alpha},\rho >0$~\text{such that}~$I_{\lambda, \mu}(u)\geq  \alpha$~\text {for any} ~${\|{u}\|}=\rho$;

$(I_3)$~\text{there exists}~${e}\in X$~\text{such that}~${\|{e}\|}>\rho~$\text{and}~$I_{\lambda, \mu}(e)<0$.

\end{Lem}

\begin{proof}
Clearly, the item $(I_1)$ holds.

As for item $(I_2)$, we derive from the embeddings $X \hookrightarrow {L^{p}_{K}}({\mathbb{R}}^{N}_{+})$, $X \hookrightarrow {L^{2^{*}}_{K}}({\mathbb{R}}^{N}_{+})$ and $X \hookrightarrow {L^{2_{*}}_{K}}({\mathbb{R}}^{N-1})$ that there exist constants $C_{1}$, $C_{2}$, $C_{3}>0$ such that for any $u\in X$,
\begin{equation*}
I_{\lambda, \mu}(u)
\geq {\frac{1}{2}{\|{u}\|}^{2}}-\mu C_{1} {\|{u}\|}^{p}-\lambda C_{2} {\|{u}\|}^{2^{*}}-\sqrt{\lambda} C_{3} {\|{u}\|}^{2_{*}}.
\end{equation*}
Setting $\theta:=\min\{p,2_*\}>2$ and $C_4:=\mu C_{1}+\lambda C_{2}+\sqrt{\lambda} C_{3}>0$, we obtain that for any $u\in X$ with $\|{u}\|\leq1$,
\begin{equation*}
I_{\lambda, \mu}(u)
\geq {\|{u}\|}^{\theta}\Big(\frac{1}{2} {\|{u}\|}^{2-\theta}-C_{4}\Big).
\end{equation*}
Therefore, the item $(I_2)$ holds for $\rho=\min\{(2+2C_{4})^\frac{1}{2-\theta},1\}$ and $\alpha=\rho^\theta>0.$

Now we are going to check item $(I_3)$. For any $u\in X$ with ${\|{u_{+}}\|}_{L^{{2}^{*}}_{K}({\mathbb{R}}^{N}_{+})}\neq 0$, $t\geq 0$, we have
\begin{equation*}
I_{\lambda, \mu}(tu)={\frac{t^{2}}{2}}{\|{u}\|}^{2}-{\frac{\mu{t^{p}}}{p}} {\|{u_{+}}\|}_{{L_{K}^{p}}({\mathbb{R}}^{N}_{+})}^{p}-{\frac{\lambda t^{2^{*}} }{2^{*}}}{\|{u_{+}}\|}_{L^{{2}^{*}}_{K}({\mathbb{R}}^{N}_{+})}^{2^{*}}-{\frac{\sqrt{\lambda}t^{2_{*}}} {2_{*}}}{\|{u_{+}}\|}_{L^{{2}_{*}}_{K}({\mathbb{R}}^{N-1})}^{2_{*}}.
\end{equation*}
Since $\lambda>0$, we obtained that $I_{\lambda, \mu}(tu)\rightarrow -\infty$ as $t\rightarrow +\infty$. Hence, we can set $e:=tu$ with $t>0$ large enough to get item $(I_3)$.
\end{proof}

Define
\begin{equation}\label{2.7}
c_{\lambda, \mu}:={\inf\limits_{\gamma \in \Gamma}}{\max\limits_{t\in[0,1]}}\ I_{\lambda, \mu}(\gamma(t)),
\end{equation}
where
\begin{equation}
\Gamma:=\big\{\gamma \in C([0,1],X):\gamma(0)=0, I_{\lambda, \mu}(\gamma(1))<0\big\}.\notag
\end{equation}
As a consequence of Lemma \ref{lem2.3}, we easily get $c_{\lambda, \mu}>0$. Next, we verify the level value $c_{\lambda, \mu}$ is in an interval where the $(PS)$ condition holds.

\begin{Lem}\label{lem2.5}
The functional $I_{\lambda, \mu}(u)$ satisfies the $(PS)_{c}$ condition at the level $c_{\lambda, \mu}$ if
 \begin{equation}
c_{\lambda, \mu}<A_\lambda,   \quad\\\label{2.8}
\end{equation}
where $A_\lambda$ and $c_{\lambda, \mu}$ are given by \eqref{2.5}, \eqref{2.7}, respectively.
\end{Lem}

\begin{proof}
From Theorem 2 in \cite{12} and Lemma \ref{lem2.3}, there exists a $(PS)_{c}$ sequence $\{u_{n}\}\subset X$ of $I_{\lambda, \mu}$ with $c=c_{\lambda, \mu}\in (0, A_{\lambda})$, that is,  for any ${\psi}\in X$,
\begin{equation}\label{2.08}
	\begin{aligned}
		I_{\lambda, \mu}({u_{n}})
		&={\frac{1}{2}}{\|{u_{n}}\|}^{2}-{\frac{\mu }{p}}{\|{(u_{n})_{+}}\|}_{{L_{K}^{p}}({\mathbb{R}}^{N}_{+})}^{p}-{\frac{\lambda }{2^{*}}}{\|(u_{n})_{+}\|}_{L^{{2}^{*}}_{K}({\mathbb{R}}^{N}_{+})}^{2^{*}}-{\frac{\sqrt{\lambda} }{2_{*}}}{\|(u _{n})_{+}\|}_{L^{{2}_{*}}_{K}({\mathbb{R}}^{N-1})}^{2_{*}} \\
		&=c_{\lambda, \mu}+{o_n}(1)
	\end{aligned}
\end{equation}
and
\begin{equation}\label{2.9}
\begin{aligned}
\langle {I'_{\lambda, \mu}(u_{n}),{\psi}}\rangle
&=\int_{\mathbb{R}^{N}_{+}}\big( K(x){\nabla{u_{n}}}{\nabla{{\psi}}}-\mu K(x)(u_{n})_{+}^{p-1}{\psi}-\lambda K(x)(u_{n})_{+}^{2^{*}-1}{\psi}\big)dx\\
&\ \ \ \ -\sqrt{\lambda}{\int_{\mathbb{R}^{N-1}}K(x{'},0)(u_{n})_{+}^{2_{*}-1}{\psi}dx{'}} =o_{n}(1)\|\psi\|.
\end{aligned}
\end{equation}
If $p\in (2,2_{*}]$, taking ${\psi}=u_{n}$, we infer from \eqref{2.08} and \eqref{2.9} that
\begin{equation}\label{q2}
\begin{aligned}
c_{\lambda, \mu}+{o_n}(1)+o_n(1)\|u_{n}\|
&=I_{\lambda, \mu}(u_{n})-{\frac{1 }{p}}\langle {I'_{\lambda, \mu}(u_{n}),{u_{n}}}\rangle \\
&=\Big(\frac{1}{2}-\frac{1}{p}\Big){\|{u_{n}}\|}^{2}+\lambda\Big(\frac{1 }{p}-\frac{1 }{2^{*}}\Big){\|(u_{n})_{+}\|}_{L^{{2}^{*}}_{K}({\mathbb{R}}^{N}_{+})}^{2^{*}} \\
& \ \ \ \ + \sqrt{\lambda}\Big({{\frac{1 }{p}}-{\frac{1 }{2_{*}}}}\Big) \|(u_{n})_{+}\|_{L^{{2}_{*}}_{K}({\mathbb{R}}^{N-1})}^{2_{*}} \\
&\geq\Big({{\frac{1}{2}}-{\frac{1}{p}}}\Big)\|{u_{n}}\|^{2}.
\end{aligned}
\end{equation}
If $p\in (2_{*},2^{*})$, 
we get
\begin{equation}\label{q3}
\begin{aligned}
c_{\lambda, \mu}+o_n(1)+o_n(1)\|u_{n}\|&=I_{\lambda, \mu}(u_{n})-{\frac{1 }{2_{*}}}\langle {I'_{\lambda, \mu}(u_{n}),{u_{n}}}\rangle \\
&=\Big({{\frac{1}{2}}-{\frac{1}{{2_{*}}}}}\Big){\|{u_{n}}\|}^{2}+\mu\Big({{\frac{1 }{{2_{*}}}}-{\frac{1 }{p}}}\Big){\|(u_{n})_{+}\|}_{L^{p}_{K}({\mathbb{R}}^{N}_{+})}^{p}\\
& \ \ \ \ +
{\lambda}\Big({{\frac{1 }{{2_{*}}}}-{\frac{1 }{2^{*}}}}\Big){\|(u_{n})_{+}\|}_{L^{{2}^{*}}_{K}({\mathbb{R}}^{N}_{+})}^{2^{*}}\\
&\geq\Big({{\frac{1}{2}}-{\frac{1}{{2_{*}}}}}\Big){\|{u_{n}}\|}^{2}.
\end{aligned}
\end{equation}
It follows from \eqref{q2} and \eqref{q3} that $\{u_{n}\}$ is bounded in $X$. Up to a subsequence if necessary, we may assume 
that for some $u\in X$,
	\begin{equation*}
		\begin{aligned}
			u_n &\rightharpoonup u \quad\quad \quad  \mbox{weakly in} ~ X,\\
			u_n &\rightarrow u\quad  \hspace{0.78cm} \mbox{strongly in} ~ L_{K}^{p} ({\mathbb{R}}^{N}_{+}),~2< p<2^*,\\
            u_n^{2^{*}-1} &\rightharpoonup u^{2^{*}-1} \quad \ \mbox{weakly in} ~ {L^{{\frac{{2}^{*}}{{2}^{*}-1}}}_{K}({\mathbb{R}}^{N}_{+})},\\
             u_n^{2_{*}-1} &\rightharpoonup u^{2_{*}-1}\quad  \
             \mbox{weakly in} ~ {L^{\frac{{2}_{*}}{{2}_{*}-1}}_{K}({\mathbb{R}}^{N-1})}.
		\end{aligned}
	\end{equation*}
Passing to the limit as $n\rightarrow \infty$ in \eqref{2.9}, we deduce that for any ${\psi}\in X$,
\begin{equation*}
\langle {I'_{\lambda, \mu}(u), {\psi}}\rangle =0.
\end{equation*}
Therefore, $u$ is a critical point of $I_{\lambda, \mu}$ and
\begin{equation*}\label{2.009}
	\begin{aligned}
I_{\lambda, \mu}({u})
&=\mu\Big(\frac{1 }{2}-{\frac{1 }{p}}\Big){\|u_{+}\|}_{L^{p}_{K}({\mathbb{R}}^{N}_{+})}^{p}
+\lambda\Big(\frac{1 }{2}-{\frac{1}{2^{*}}}\Big){\|u_{+}\|}_{L^{{2}^{*}}_{K}({\mathbb{R}}^{N}_{+})}^{2^{*}}\\
&\quad+\sqrt{\lambda}\Big(\frac{1 }{2}-{\frac{1}{2_{*}}}\Big){\|u_{+}\|}_{L^{{2}_{*}}_{K}({\mathbb{R}}^{N-1})}^{2_{*}}\geq 0.
\end{aligned}
\end{equation*}

Set $z_{n}:=u_{n}-u$. From Brezis-Lieb Lemma \cite{13}, we have
\begin{equation*}
\begin{aligned}
{\|{(u_{n})_{+}}\|}_{L^{{2}^{*}}_{K}({\mathbb{R}}^{N}_{+})}^{2^{*}}
&= {\|{(z_{n})_{+}}\|}_{L^{{2}^{*}}_{K}({\mathbb{R}}^{N}_{+})}^{2^{*}}+
{\|{u_+}\|}_{L^{{2}^{*}}_{K}({\mathbb{R}}^{N}_{+})}^{2^{*}}+o_n(1),\\
{\|{(u_{n})_{+}}\|}_{L^{{2}_{*}}_{K}({\mathbb{R}}^{N-1})}^{2_{*}}
&= {\|{(z_{n})_{+}}\|}_{L^{{2}_{*}}_{K}({\mathbb{R}}^{N-1})}^{2_{*}}+
{\|{u_+}\|}_{L^{{2}_{*}}_{K}({\mathbb{R}}^{N-1})}^{2_{*}}+o_n(1).
\end{aligned}
\end{equation*}
Clearly,
\begin{equation*}
{\|u_{n}\|}^{2}={\|{z_{n}}\|}^{2}+{\|{u}\|}^{2}+o_n(1).\notag
\end{equation*}
Thus, we obtain from \eqref{2.08} and \eqref{2.9} that
\begin{equation}\label{2.010}
\begin{aligned}
I_{\lambda, \mu}({u})+{\frac{1}{2}}{\|{z_{n}}\|}^{2}-{\frac{\lambda }{2^{*}}}{\|(z_{n})_{+}\|}_{L^{{2}^{*}}_{K}({\mathbb{R}}^{N}_{+})}^{2^{*}}-{\frac{\sqrt{\lambda} }{2_{*}}}{\|(z _{n})_{+}\|}_{L^{{2}_{*}}_{K}({\mathbb{R}}^{N-1})}^{2_{*}}
=c_{\lambda, \mu}+o_n(1)
\end{aligned}
\end{equation}
and
\begin{equation}\label{2.011}
{\|{z_{n}}\|}^{2}-\lambda{\|(z_{n})_{+}\|}_{L^{{2}^{*}}_{K}({\mathbb{R}}^{N}_{+})}^{2^{*}}-\sqrt{\lambda}{\|(z _{n})_{+}\|}_{L^{{2}_{*}}_{K}({\mathbb{R}}^{N-1})}^{2_{*}}=o_n(1).
\end{equation}

Next, we show that there exists a subsequence of $\{z_{n}\}$, still denoted by $\{z_{n}\}$ such that ${\|{z_{n}}\|}\rightarrow 0$ as $n\rightarrow \infty$. By contradiction, we assume that there exists $\beta>0$ such that ${\|{z_{n}}\|}\geq \beta>0$ for any $n\in \mathbb{N}$. Let
\begin{equation*}
\Hat I_{\lambda, \mu }(u):={\frac{1}{2}}{\|{u}\|}^{2}-{\frac{\lambda }{2^{*}}}{\|{u_{+}}\|}_{L^{{2}^{*}}_{K}({\mathbb{R}}^{N}_{+})}^{2^{*}}-{\frac{\sqrt{\lambda} }{2_{*}}}{\|{u_{+}}\|}_{L^{{2}_{*}}_{K}({\mathbb{R}}^{N-1})}^{2_{*}}  .\notag
\end{equation*}
We claim that
\begin{equation}\label{2.13}
\sup\limits_{t>0}\Hat I_{\lambda, \mu }(tz_{n})\geq A_{\lambda}-\epsilon, ~\text{for}~n~\text{large enough},
\end{equation}
where $\epsilon$ is a small positive constant.

In fact,
 suppose on the contrary that there exists $n_{0} \in \mathbb{N}$ so that $\sup\limits_{t>0}\Hat I_{\lambda, \mu }({tz_{n}})< A_{\lambda}-\epsilon$ for $n\geq n_{0}$. Since for any $0<b<\infty$,
\begin{equation*}
\sup\limits_{t>0}\Hat I_{\lambda, \mu }(tz_{n})=\sup\limits_{t>0}\Hat I_{\lambda, \mu }(tbz_{n}),\notag
\end{equation*}
 there exists $0<b_{0}<\infty$ with
\begin{equation}\label{2.14}
{\|\nabla v_{n}\|_{{L_{K}^{2}}(\mathbb{R}^{N}_{+})}}={\|\nabla{{\varphi}_{{\varepsilon},{0}}}\|_{{L^{2}}(\mathbb{R}^{N}_{+})}},
\end{equation}
such that $\sup\limits_{t>0}\Hat I_{\lambda, \mu }(tv_{n})< A_\lambda-\epsilon$, where $v_{n}=b_{0}{z_{n}}$.

If one has
\begin{equation}\label{2.115}
\frac{{\|(v_{n})_{+}\|_{L^{{2}^{*}}_{K}({\mathbb{R}}^{N}_{+})}^{2}}}{{\|(v_{n})_{+}\|_{L^{{2}_{*}}_{K}({\mathbb{R}}^{N-1})}^{2}}}\geq \frac{{\|{{{\varphi}_{{\varepsilon},{0}}}}\|_{L^{{2}^{*}}({\mathbb{R}}^{N}_{+})}^{2}}}{{\|{{{\varphi}_{{\varepsilon},{0}}}}\|_{L^{{2}_{*}}({\mathbb{R}}^{N-1})}^{2}}},
\end{equation}
by  Lemma \ref{lem2.4} and the fact that ${{{\varphi}_{{\varepsilon},{0}}}}$ reaches the infimum  $S_{1}$ in \eqref{2.1}, we have that
\begin{equation}\label{15}
\frac{{\|\nabla{{v_{n}}}\|_{{L_{K}^{2}}(\mathbb{R}^{N}_{+})}^{2}}}{{\|{{v_{n}}}\|_{L^{{2}^{*}}_{K}({\mathbb{R}}^{N}_{+})}^{2}}}\geq
\frac{{\|\nabla{{{{\varphi}_{{\varepsilon},{0}}}}}\|_{{L^{2}}(\mathbb{R}^{N}_{+})}^{2}}}{{\|{{{{\varphi}_{{\varepsilon},{0}}}}}\|_{L^{{2}^{*}}({\mathbb{R}}^{N}_{+})}^{2}}}.
\end{equation}
From \eqref{2.14}-\eqref{15}, we get
$${{\|{{{{\varphi}_{{\varepsilon},{0}}}}}\|_{L^{{2}^{*}}({\mathbb{R}}^{N}_{+})}^{2}}}\geq{{\|{{v_{n}}}\|_{L^{{2}^{*}}_{K}({\mathbb{R}}^{N}_{+})}^{2}}}\geq {{\|(v_{n})_{+}\|_{L^{{2}^{*}}_{K}({\mathbb{R}}^{N}_{+})}^{2}}}$$
and
$${{\|{{{{\varphi}_{{\varepsilon},{0}}}}}\|_{L^{{2}_{*}}({\mathbb{R}}^{N-1})}^{2}}}\geq{{\|({v_{n}})_{+}\|_{L^{{2}_{*}}_{K}({\mathbb{R}}^{N-1})}^{2}}}.$$
Hence, there holds
\begin{equation*}
\sup\limits_{t>0}\Hat I_{\lambda, \mu }(tv_{n})\geq
{\sup\limits_{t>0}}\Phi_\lambda(t{{{{\varphi}_{{\varepsilon},{0}}}}})\geq A_\lambda,
\end{equation*}
which contradicts with $\sup\limits_{t>0}\Hat I_{\lambda, \mu }(tv_{n})< A_\lambda-\epsilon$.

Thus, \eqref{2.115} is not true, i.e.,
\begin{equation*}
\frac{{\|(v_{n})_{+}\|_{L^{{2}^{*}}_{K}({\mathbb{R}}^{N}_{+})}^{2}}}{{\|(v_{n})_{+}\|_{L^{{2}_{*}}_{K}({\mathbb{R}}^{N-1})}^{2}}}< \frac{{\|{{{\varphi}_{{\varepsilon},{0}}}}\|_{L^{{2}^{*}}({\mathbb{R}}^{N}_{+})}^{2}}}{{\|{{{\varphi}_{{\varepsilon},{0}}}}\|_{L^{{2}_{*}}({\mathbb{R}}^{N-1})}^{2}}}.
\end{equation*}
Moreover, 
\begin{equation*}
\frac{{\|{{{\varphi}_{{\varepsilon},{{\tau}}}}}\|_{L^{{2}^{*}}({\mathbb{R}}^{N}_{+})}^{2}}}{{\|{{{\varphi}_{{\varepsilon},{{\tau}}}}}\|_{L^{{2}_{*}}({\mathbb{R}}^{N-1})}^{2}}}\rightarrow 0 ~\text{as}~ {\tau}\rightarrow \infty.\notag
\end{equation*}
Therefore, there must be some ${{\tau}_{0}}>0$ such that
\begin{equation*}
\frac{{\|(v_{n})_{+}\|_{L^{{2}^{*}}_{K}({\mathbb{R}}^{N}_{+})}^{2}}}{{\|(v_{n})_{+}\|_{L^{{2}_{*}}_{K}({\mathbb{R}}^{N-1})}^{2}}}= \frac{{\|{{{\varphi}_{{\varepsilon},{{{\tau}_{0}}}}}}\|_{L^{{2}^{*}}({\mathbb{R}}^{N}_{+})}^{2}}}{{\|{{{\varphi}_{{\varepsilon},{{{\tau}_{0}}}}}}\|_{L^{{2}_{*}}({\mathbb{R}}^{N-1})}^{2}}}.
\end{equation*}
Let
\begin{equation}\label{2.117}
k:=\frac{{\|(v_{n})_{+}\|_{L^{{2}^{*}}_{K}({\mathbb{R}}^{N}_{+})}}}{{\|{{{\varphi}_{{\varepsilon},{{{\tau}_{0}}}}}}\|_{L^{{2}^{*}}({\mathbb{R}}^{N}_{+})}}}>0,
\end{equation}
then
\begin{equation}\label{2.118}
k=\frac{{\|(v_{n})_{+}\|_{L^{{2}_{*}}_{K}({\mathbb{R}}^{N-1})}}}{{\|{{{\varphi}_{{\varepsilon},{{{\tau}_{0}}}}}}\|_{L^{{2}_{*}}({\mathbb{R}}^{N-1})}}}
\end{equation}
and
\begin{equation*}
\frac{{\|v_{n}\|_{L^{{2}^{*}}_{K}({\mathbb{R}}^{N}_{+})}}}{{\|{{{\varphi}_{{\varepsilon},{{{\tau}_{0}}}}}}\|_{L^{{2}^{*}}({\mathbb{R}}^{N}_{+})}}}\geq k,~\frac{{\|v_{n}\|_{L^{{2}_{*}}_{K}({\mathbb{R}}^{N-1})}}}{{\|{{{\varphi}_{{\varepsilon},{{{\tau}_{0}}}}}}\|_{L^{{2}_{*}}({\mathbb{R}}^{N-1})}}}\ \geq k.\notag
\end{equation*}
By Lemma \ref{lem2.4} and the fact that $\varphi_{\varepsilon, \tau_{0}}$ reaches the infimum  $S_{{\theta}_{0}}$ in \eqref{2.1}, we have
\begin{equation*}
\frac{{\|\nabla{{v_{n}}}\|_{{L_{K}^{2}}(\mathbb{R}^{N}_{+})}^{2}}}{{{\theta}_{0}} {{\|{{v_{n}}}\|_{L^{{2}^{*}}_{K}({\mathbb{R}}^{N}_{+})}^{2}}}+(1-{{\theta}_{0}}){{\|{{v_{n}}}\|_{L^{{2}_{*}}_{K}({\mathbb{R}}^{N-1})}^{2}}}}\geq \frac{{\|\nabla{{{{\varphi}_{{\varepsilon},{{{\tau}_{0}}}}}}}\|_{{L^{2}}(\mathbb{R}^{N}_{+})}^{2}}}{{{\theta}_{0}} {{\|{{{{\varphi}_{{\varepsilon},{{{\tau}_{0}}}}}}}\|_{L^{{2}^{*}}({\mathbb{R}}^{N}_{+})}^{2}}}+(1-{{\theta}_{0}}){{\|{{{{\varphi}_{{\varepsilon},{{{\tau}_{0}}}}}}}\|_{L^{{2}_{*}}({\mathbb{R}}^{N-1})}^{2}}}}.
\end{equation*}
It follows that
\begin{equation}\label{2.119}
\frac{{\|\nabla{{v_{n}}}\|_{{L_{K}^{2}}(\mathbb{R}^{N}_{+})}^{2}}}{{\|\nabla{{{{\varphi}_{{\varepsilon},{{{\tau}_{0}}}}}}}\|_{{L^{2}}(\mathbb{R}^{N}_{+})}^{2}}}\geq k^{2}.
\end{equation}
On the one hand,
\begin{equation*}
\sup\limits_{t>0}\Hat I_{\lambda, \mu }(tv_{n})
=\lambda\Big({\frac{1}{2}}-{\frac{1 }{2^{*}}}\Big){t_{1}^{2^{*}}}{\|{(v_{n})_{+}}\|}_{L^{{2}^{*}}_{K}({\mathbb{R}}^{N}_{+})}^{2^{*}}+\sqrt{\lambda}\Big({\frac{1}{2}}-\frac{1 }{2_{*}}\Big){t_{1}^{2_{*}}}{\|{(v_{n})_{+}}\|_{L^{{2}_{*}}_{K}({\mathbb{R}}^{N-1})}^{2_{*}}}  ,\notag
\end{equation*}
where $t_{1}>0$ satisfies
\begin{equation*}
\lambda t_{1}^{2^{*}-2}{{\|{(v_{n})_{+}}\|_{L^{{2}^{*}}_{K}({\mathbb{R}}^{N}_{+})}^{2^{*}}}}+
\sqrt{\lambda}t_{1}^{2_{*}-2}{{\|{(v_{n})_{+}}\|_{L^{{2}_{*}}_{K}({\mathbb{R}}^{N-1})}^{2_{*}}}}-{{\|\nabla{{v_{n}}}\|_{{L_{K}^{2}}
(\mathbb{R}^{N}_{+})}^{2}}}=0.\notag
\end{equation*}
Thus for $0<t\leq t_{1}$, we have
\begin{equation}\label{2.220}
\lambda t^{2^{*}-2}{{\|{(v_{n})_{+}}\|_{L^{{2}^{*}}_{K}({\mathbb{R}}^{N}_{+})}^{2^{*}}}}+\sqrt{\lambda}t^{2_{*}-2}{{\|{(v_{n})_{+}}\|_{L^{{2}_{*}}_{K}({\mathbb{R}}^{N-1})}^{2_{*}}}}\leq {{\|\nabla{{v_{n}}}\|_{{L_{K}^{2}}(\mathbb{R}^{N}_{+})}^{2}}}.
\end{equation}
On the other hand, by \eqref{2.117} and \eqref{2.118}, we have
\begin{equation}\label{2.221}
	\begin{aligned}
		\sup\limits_{t>0}\Phi_{\lambda}(t{{{\varphi_{\varepsilon, \tau_{0}}}}})
		&=\lambda\Big({\frac{1}{2}}-{\frac{1 }{2^{*}}}\Big){t_{2}^{2^{*}}}{\|{\varphi_{\varepsilon,{{{{\tau}_{0}}}}}}\|}_{L^{{2}^{*}}({\mathbb{R}}^{N}_{+})}^{2^{*}}
		+\sqrt{\lambda}\Big({\frac{1}{2}}-{\frac{1 }{2_{*}}}\Big){t_{2}^{2_{*}}}{\|{{{\varphi_{{\varepsilon}, \tau_{0}}}}}\|}_{L^{{2}_{*}}(\mathbb{R}^{N-1})}^{2_{*}}\\
		&=\lambda\Big({\frac{1}{2}}-{\frac{1 }{2^{*}}}\Big)\Big(\frac{t_{2}}{k}\Big)^{2^{*}}{\|{(v_{n})_{+}}\|}_{L^{{2}^{*}}_{K}({\mathbb{R}}^{N}_{+})}^{2^{*}}
		+\sqrt{\lambda}\Big({\frac{1}{2}}-{\frac{1 }{2_{*}}}\Big)(\frac{t_{2}}{k})^{2_{*}}{\|{(v_{n})_{+}}\|}_{L^{{2}_{*}}_{K}({\mathbb{R}}^{N-1})}^{2_{*}},
	\end{aligned}
\end{equation}
where $t_{2}>0$ is the root of
\begin{equation*}
\lambda t^{2^{*}-2}{{\|\varphi_{\varepsilon, \tau_{0}}\|_{L^{2^{*}} ({\mathbb{R}}^{N}_{+})}^{2^{*}}}}+\sqrt{\lambda}t^{2_{*}-2}{{\|{{{{\varphi_{\varepsilon,\tau_{0}}}}}}\|_{L^{{2}_{*}}({\mathbb{R}}^{N-1})}^{2_{*}}}}-{{\|\nabla{{{{\varphi}_{{\varepsilon},{{{{\tau}_{0}}}}}}}}\|_{{L^{2}}(\mathbb{R}^{N}_{+})}^{2}}}=0.\notag
\end{equation*}
Using \eqref{2.117}-\eqref{2.119} again, we obtain
\begin{equation*}
\lambda\Big(\frac{t_{2}}{k}\Big)^{2^{*}-2}{{\|{(v_{n})_{+}}\|_{L^{{2}^{*}}_{K}({\mathbb{R}}^{N}_{+})}^{2^{*}}}}+\sqrt{\lambda}\Big(\frac{t_{2}}{k}\Big)^{2_{*}-2}{{\|{(v_{n})_{+}}\|_{L^{{2}_{*}}_{K}({\mathbb{R}}^{N-1})}^{2_{*}}}}\leq{{\|\nabla{{v_{n}}}\|_{{L_{K}^{2}}(\mathbb{R}^{N}_{+})}^{2}}}.
\end{equation*}
From \eqref{2.220}, we get
\begin{equation}\label{20}
\frac{t_{2}}{k}\leq t_{1}.
\end{equation}
Substituting \eqref{20} into \eqref{2.221}, we get
\begin{equation*}
\sup\limits_{t>0}\Phi_\lambda(t{{{{\varphi}_{{\varepsilon},{{{{\tau}_{0}}}}}}}})\leq
\sup\limits_{t>0}\Hat I_{\lambda, \mu }(tv_{n}),\notag
\end{equation*}
which also contradicts with $\sup\limits_{t>0}\Hat I_{\lambda, \mu }(tv_{n})< A_\lambda-\epsilon$. Hence, our claim (\ref {2.13}) holds if $\|z_n\| \ge \beta >0$.

From \eqref{2.011}, we have
\begin{equation}\label{c2.27}
\hat I_{\lambda, \mu }(z_{n})=\sup\limits_{t>0}\hat I_{\lambda, \mu }(tz_{n})+o_{n}(1).
\end{equation}
Combining \eqref{2.010} with \eqref{c2.27} yields that
\begin{equation*}
I_{\lambda, \mu}({u_{n}})=I_{\lambda, \mu}(u)+\sup\limits_{t>0}\hat I_{\lambda, \mu }(tz_{n})+o_{n}(1)=c_{\lambda, \mu}+o_n(1).\notag
\end{equation*}
From \eqref{2.13} and $I_{\lambda, \mu}(u)\geq 0$, we have that $c_{\lambda, \mu}\geq A_\lambda-\epsilon$ provided $n\geq n_{0}$, which contradicts \eqref{2.8} as $\epsilon$ is small enough. Therefore, ${\|{z_{n}}\|}\rightarrow 0$, that is, $I_{\lambda, \mu}(u)$ satisfies the $(PS)_{c}$ condition for $c=c_{\lambda, \mu}\in(0, A_\lambda)$.
\end{proof}

\vs{3mm}
\section{The proof of Theorem \ref{Th1.1}}\label{S3}
\vs{2mm}

In the previous section, we have proved a local compactness result for the functional $I_{\lambda, \mu}(u)$ under the condition \eqref{2.8}. In this section, we devote to verifying condition \eqref{2.8} and then complete the proof of Theorem \ref{Th1.1}.

Set
\begin{equation*}
c_{\lambda, \mu}^{*}:=\inf\limits_{u\in X\setminus\{0\}}\sup\limits_{t>0}I_{\lambda, \mu}(tu),
\end{equation*}
then $c_{\lambda, \mu}\leq c_{\lambda, \mu}^{*}$. Hence, the condition \eqref{2.8} in Lemma \ref{lem2.5} holds if
\begin{equation}\label{3.01}
c_{\lambda, \mu}^{*}<A_{\lambda}=\lambda^{-\frac{N-2}{2}}A,
\end{equation}
where $A_{\lambda}$ and $A$ are given by (\ref{2.5}), \eqref{2000}, respectively.
Let $\phi \in C_0^{\infty}({\mathbb{R}}^{N}_{+},[0,1])$ be a cut-off function such that $\phi \equiv 1$ in $B_{1}(0)\cap {\mathbb{R}}^{N}_{+}$ and $\phi\equiv 0$ in $\overline{{{\mathbb{R}}^{N}_{+}}}\setminus B_{2}(0)$, and define
\begin{eqnarray*}
&&\tilde{\psi}_{\lambda,\varepsilon}(x)
:=K(x)^{-\frac{1}{2}}\phi(x)\psi_{\lambda,\varepsilon}(x),\\
&& \ \ \tilde {U}_{{\varepsilon}}(x):=K(x)^{-\frac{1}{2}}\phi(x)U_{{\varepsilon}}(x),\notag
\end{eqnarray*}
where $\psi_{\lambda,\varepsilon}$ and $U_{{\varepsilon}}$ are defined in Lemma \ref{lem2.2} and \eqref{2.111}.

Noting that $\tilde{\psi}_{\lambda,\varepsilon}=\lambda^{-\frac{N-2}{4}}\tilde{U}_{{\varepsilon}}$, we have
\begin{eqnarray*}
&&{\|\tilde{\psi}_{\lambda,\varepsilon}\|}^{2}=\lambda^{-\frac{N-2}{2}}{\|\tilde{U}_{{\varepsilon}}\|}^{2},\\
&&{\|\tilde{\psi}_{\lambda,\varepsilon}\|}_{L^{p}_{K}({\mathbb{R}}^{N}_{+})}^{p}=\lambda^{-\frac{(N-2)p}{4}}{\|\tilde{U}_{{\varepsilon}}\|}_{L^{p}_{K}({\mathbb{R}}^{N}_{+})}^{p},\\
&&\lambda{\|\tilde{\psi}_{\lambda,\varepsilon}\|}_{L^{{2}^{*}}_{K}({\mathbb{R}}^{N}_{+})}^{2^{*}}=\lambda^{-\frac{N-2}{2}}{\|\tilde{U}_{{\varepsilon}}\|}_{L^{{2}^{*}}_{K}({\mathbb{R}}^{N}_{+})}^{2^{*}},\\
&&\sqrt{\lambda}{\|\tilde{\psi}_{\lambda,\varepsilon}\|}_{L^{{2}_{*}}_{K}({\mathbb{R}}^{N-1})}^{2_{*}}=\lambda^{-\frac{N-2}{2}}{\|\tilde{U}_{{\varepsilon}}\|}_{L^{{2}_{*}}_{K}({\mathbb{R}}^{N-1})}^{2_{*}},
\end{eqnarray*}
where $p\in(2,2^*)$. Thus, one has
\begin{equation*}
{I}_{\lambda, \mu }(\tilde{\psi}_{\lambda,\varepsilon})=\lambda^{-\frac{N-2}{2}}\tilde{I}_{\lambda, \mu }(\tilde{U}_{{\varepsilon}}),
\end{equation*}
where
\begin{equation*}
\tilde{I}_{\lambda, \mu }(u):={\frac{1}{2}}{\|{u}\|}^{2}-\frac{\mu{\lambda^{\frac{(N-2)(2-p)}{4}}      }}{p}{\|{u_+}\|}_{L^{p}_{K}({\mathbb{R}}^{N}_{+})}^{p}-{\frac{1 }{2^{*}}}{\|{u_+}\|}_{L^{{2}^{*}}_{K}({\mathbb{R}}^{N}_{+})}^{2^{*}}
-\frac{1} {2_{*}}{\|{u_+}\|}_{L^{{2}_{*}}_{K}({\mathbb{R}}^{N-1})}^{2_{*}}.
\end{equation*}

In order to verify \eqref{3.01}, it suffices to verify
\begin{equation}\label{0}
\sup\limits_{t>0}\tilde{I}_{\lambda, \mu }(t\tilde{U}_{{\varepsilon}})<A
\end{equation}
under the assumptions of Theorem \ref{Th1.1}.
For simplicity, set $k_{N}:=\big(N(N-2)\big)^{\frac{N-2}{4}}$, then
\begin{equation*}
{U}_{\varepsilon}(x)=\frac{{\varepsilon}^{\frac {N-2}{2}}k_{N}}{({\varepsilon}^{2}+|x{'}|^{2}+|x_{N}+{\varepsilon}x_{N}^{0}|^{2})^{\frac{N-2}{2}}},\notag\\
\end{equation*}
where $x_{N}^{0}$ is given by \eqref{2.2}.


Before verifying \eqref{0}, we first present some crucial estimates for $\tilde {U}_{{\varepsilon}}(x) $.
\begin{Lem}\label{lem3.1}
There holds
\begin{equation*}\label{3.09}
{\|\tilde{U}_{{\varepsilon}}\|}^{2}=
\begin{cases}
K_{1}+\alpha_{N}{{\varepsilon}}^{2}+o({{\varepsilon}}^{2}),~&~~N\geq 5,\\[1mm]
K_{1}+\frac{k_{4}^{2}\omega_4}{2}{{\varepsilon}}^{2}|\ln{\varepsilon}|+O({\varepsilon}^{2}),~&~~N=4,\\[1mm]
K_1+O({\varepsilon}),~&~~N=3,\\[1mm]
\end{cases}
\end{equation*}
where $\varepsilon>0$ is sufficiently small and
\begin{equation*}\label{4.01}
\alpha_{N}=\frac{(N-2)k_{N}^{2}}{2}\int_{\mathbb{R}^{N}_{+}}\frac{|y{'}|^{2}+{y_{N}(y_{N}+x_{N}^{0})}}
{{{{(1+|y{'}|^{2}+|y_{N}+x_{N}^{0}|^{2})^{N-1}}}}}dy.
\end{equation*}
\end{Lem}

\begin{proof}
From the definitions of $\tilde{U}_{{\varepsilon}} $ and $\phi$, we directly compute that
\begin{equation*}
  \begin{aligned}
{\|\tilde{U}_{{\varepsilon}}    \|}^{2}
&=\int_{\mathbb{R}^{N}_{+}}\Big(|\nabla{\phi}|^{2}U_\varepsilon^{2}+\notag
2{\phi}{{{U}_{\varepsilon}}}({\nabla{\phi}}\cdot{\nabla{{{{U}_{\varepsilon}}}}})\notag
-\frac{1}{2}{\phi}U_\varepsilon^{2}(x\cdot{\nabla{\phi}})\Big)dx\\ &\quad+\int_{\mathbb{R}^{N}_{+}}{\phi}^{2}|\nabla{{{U}_{\varepsilon}}}|^{2}dx-\notag
\frac{1}{2}\int_{\mathbb{R}^{N}_{+}}{\phi}^{2}{{U}_{\varepsilon}}(x\cdot{\nabla{{{U}_{\varepsilon}}}})dx+\frac{1}{16}\int_{\mathbb{R}^{N}_{+}}{\phi}^{2}|x|^{2}U_\varepsilon^{2}dx \notag
  \end{aligned}
\end{equation*}
and
\begin{equation*}
\int_{\mathbb{R}^{N}_{+}}{|\nabla{\phi}|^{2}U_\varepsilon^{2}}dx={{\varepsilon}^{{N-2}}}\int_{ B_{2}^{+}\setminus{B_{1}^{+}}}{\frac{|\nabla{\phi}|^{2}{k_{N}^{2}}}{({\varepsilon}^{2}+|x{'}|^{2}+|x_{N}+{\varepsilon}x_{N}^{0}|^{2})^{N-2}}}dx=O({\varepsilon}^{N-2}),\notag
\end{equation*}
where $\varepsilon >0$ small enough and $B_{r}^{+}=B_{r}(0)\cap {\mathbb{R}}^{N}_{+}$ for any $r>0$. 
Similarly arguments for the other terms in the brackets above, we obtain
\begin{equation*}
\int_{\mathbb{R}^{N}_{+}}\Big(|\nabla{\phi}|^{2}U_\varepsilon^{2}+\notag
2{\phi}{{{U}_{\varepsilon}}}({\nabla{\phi}}\cdot{\nabla{{{{U}_{\varepsilon}}}}})\notag
-\frac{1}{2}{\phi}U_\varepsilon^{2}(x\cdot{\nabla{\phi}})\Big)dx=O({\varepsilon}^{N-2}),
\end{equation*}
and thus
\begin{equation}\label{3.05}
{\|\tilde{U}_{\varepsilon}\|}^{2}
= \int_{\mathbb{R}^{N}_{+}}{\phi}^{2}|\nabla{U_{\varepsilon}}|^{2}dx-
\frac{1}{2}\int_{\mathbb{R}^{N}_{+}}\phi^{2}{{U}_{\varepsilon}}(x\cdot{\nabla{{{U}_{\varepsilon}}}})dx+\frac{1}{16}\int_{\mathbb{R}^{N}_{+}}{\phi}^{2}|x|^{2}U_\varepsilon^{2}dx+O(\varepsilon^{N-2}).
\end{equation}

To estimate each of the integrals on the right-hand side of (\ref{3.05}), we first calculate that
\begin{equation*}
{\nabla{{{U}_{\varepsilon}}}}=-\frac{{(N-2)k_{N}{\varepsilon}^{\frac {N-2}{2}}}}{{({\varepsilon}^{2}+|x{'}|^{2}+|x_{N}+{\varepsilon}x_{N}^{0}|^{2})^{\frac{N}{2}}}}(x_{1}, \cdots ,x_{N-1},x_{N}+{\varepsilon}x_{N}^{0}).\notag
\end{equation*}
From $\phi^{2}|\nabla{U_{\varepsilon}}|^{2}=|\nabla U_{\varepsilon}|^{2} +(\phi^{2}-1)|\nabla{U_{\varepsilon}}|^{2}$ and \eqref{2.111}, we get that
\begin{equation}\label{3.06}
	\begin{aligned}
		\int_{\mathbb{R}^{N}_{+}}{\phi}^{2}|\nabla{{{U}_{\varepsilon}}}|^{2}dx
		&=
		K_{1}+{(N-2)^{2}k_{N}^{2}}{\varepsilon}^{N-2}\int_{ {\mathbb{R}^{N}_{+}}\setminus{B_{1}^{+}}} \frac{({\phi}^{2}-1)(|x{'}|^{2}+|x_{N}+{\varepsilon}x_{N}^{0}|^{2})}
		{{{({\varepsilon}^{2}+|x{'}|^{2}+|x_{N}+{\varepsilon}x_{N}^{0}|^{2})^{N}}}}dx\\
		&=K_{1}+O({\varepsilon}^{N-2}),
	\end{aligned}
\end{equation}
where we used the fact that the integral
\begin{equation*}
	\int_{ {\mathbb{R}^{N}_{+}}\setminus{B_{1}^{+}}}\frac{|x{'}|^{2}+|x_{N}+{\varepsilon}x_{N }^{0}|^{2}}{{{ (\varepsilon^{2}+|x'|^{2}+|x_{N}+{\varepsilon}x_{N}^{0}|^{2})^{N}}}}dx
\end{equation*}
is finite as $N\geq 3$.

Now we are going to estimate the second integral in the right-hand side of (\ref {3.05}).
 Using the same arguments, 
there holds for $N\geq 5$,
\begin{equation}\label{3.07}
 \begin{aligned}
& \ \ \ \ \int_{\mathbb{R}^{N}_{+}}{\phi}^{2}{{U}_{\varepsilon}}(x\cdot{\nabla{{{U}_{\varepsilon}}}})dx\\
&=\int_{\mathbb{R}^{N}_{+}}{{U}_{\varepsilon}}(x\cdot{\nabla{{{U}_{\varepsilon}}}})dx+O({\varepsilon}^{N-2})\\
&=-(N-2)k_{N}^{2}{\varepsilon}^{N-2}\int_{\mathbb{R}^{N}_{+}}\frac{|x{'}|^{2}+x_{N}(x_{N}+{\varepsilon}x_{N}^{0})}{{{{({\varepsilon}^{2}+|x{'}|^{2} +|x_{N}+{\varepsilon}x_{N}^{0}|^{2})^{N-1}}}}}dx+O({\varepsilon}^{N-2})\\
&=-(N-2)k_{N}^{2}{\varepsilon}^{2}\int_{\mathbb{R}^{N}_{+}}\frac{|y{'}|^{2}+y_{N}(y_{N}+x_{N}^{0})}{{{{(1+|y{'}|^{2}+|y_{N}+x_{N}^{0}|^{2})^{N-1}}}}}dy+O({\varepsilon}^{N-2})\\
&=-(N-2)k_{N}^{2}C_{1, N}{\varepsilon}^{2}+O({\varepsilon}^{N-2}),
 \end{aligned}
\end{equation}
 where
\begin{equation}\label{adc}
	C_{1,N}=\int_{\mathbb{R}^{N}_{+}}\frac{|y{'}|^{2}+{y_{N}(y_{N}+x_{N}^{0})}}
	{{{{(1+|y{'}|^{2}+|y_{N}+x_{N}^{0}|^{2})^{N-1}}}}}dy.
\end{equation}

For $N=3, 4$, we have
\begin{equation}\label{z1}
\begin{aligned}
\int_{\mathbb{R}^{N}_{+}}{\phi}^{2}{{U}_{\varepsilon}}(x\cdot{\nabla{{{U}_{\varepsilon}}}})dx
&=\int_{B_{2}^{+}}{{U}_{\varepsilon}}(x\cdot{\nabla{{{U}_{\varepsilon}}}})dx
+\int_{B_{2}^{+}\setminus{B_{1}^{+}}}({\phi}^{2}-1){{U}_{\varepsilon}}(x\cdot{\nabla{{{U}_{\varepsilon}}}})dx\\
&=-(N-2)k_{N}^{2}{\varepsilon}^{2}\int_{B_{2/\varepsilon}^{+}}\frac{|y{'}|^{2}+y_{N}(y_{N}+x_{N}^{0})}
{{{{(1+|y{'}|^{2}+|y_{N}+x_{N}^{0}|^{2})^{N-1}}}}}dy+O({\varepsilon}^{N-2})\\
&=-(N-2)k_{N}^{2}{\varepsilon}^{2}\int_{B_{2/\varepsilon}^{+}\setminus{B_{1}^{+}}}\frac{|y{'}|^{2}+y_{N}(y_{N}+x_{N}^{0})}
{{{{(1+|y{'}|^{2}+|y_{N}+x_{N}^{0}|^{2})^{N-1}}}}}dy+O(\varepsilon^2)\\
&\quad+O({\varepsilon}^{N-2}).
\end{aligned}
\end{equation}
Noting that for $N=4$,
\begin{equation*}
\begin{aligned}
\int_{B_{2/\varepsilon}^{+}\setminus B_{1}^{+}}\frac{|y{'}|^{2}+y_{4}(y_{4}+x_{4}^{0})}
{{{{(1+|y{'}|^{2}+|y_{4}+x_{4}^{0}|^{2})^{3}}}}}dy
&=\int_{B_{2/\varepsilon}^{+}\setminus{B_{1}^{+}}}\frac{|y|^{2}}
{{{{(1+|y{'}|^{2}+|y_{4}+x_{4}^{0}|^{2})^{3}}}}}dy\\
&\quad+x_{4}^{0}
\int_{B_{2/\varepsilon}^{+}\setminus{B_{1}^{+}}}\frac{y_{4}}
{{{{(1+|y'|^{2}+|y_{4}+x_{4}^{0}|^{2})^{3}}}}}dy\\
&=\int_{B_{2/\varepsilon}^{+}\setminus{B_{1}^{+}}}\frac{|y|^{2}}
{{{{(1+|y{'}|^{2}+|y_{4}+x_{4}^{0}|^{2})^{3}}}}}dy+O_\varepsilon(1)
\end{aligned}
\end{equation*}
and
\begin{equation*}
\begin{aligned}
0&<\int_{B_{2/\varepsilon}^{+}\setminus{B_{1}^{+}}}\frac{1}{|y|^{4}}dy-\int_{B_{2/\varepsilon}^{+}\setminus{B_{1}^{+}}}\frac{|y|^{2}}
{{{{(1+|y{'}|^{2}+|y_{4}+x_{4}^{0}|^{2})^{3}}}}}dy\\
&=\int_{B_{2/\varepsilon}^{+}\setminus{B_{1}^{+}}}\frac{{(1+|y{'}|^{2}+|y_{4}+x_{4}^{0}|^{2})^{3}}-|y|^{6}}{|y|^{4}{(1+|y{'}|^{2}+|y_{4}+x_{4}^{0}|^{2})^{3}}}dy=O_\varepsilon(1),
\end{aligned}
\end{equation*}
 we deduce that
\begin{equation}\label{zz2}
\begin{aligned}
&\ \ \ \ \int_{B_{2/\varepsilon}^{+}\setminus B_{1}^{+}}\frac{|y{'}|^{2}+y_{4}(y_{4}+x_{4}^{0})}
{{{{(1+|y{'}|^{2}+|y_{4}+x_{4}^{0}|^{2})^{3}}}}}dy\\
&=\int_{B_{2/\varepsilon}^{+}\setminus{B_{1}^{+}}}\frac{1}{|y|^{4}}dy+O_\varepsilon(1)
=\frac{\omega_4}{2}\int_{1}^{{{2}/\varepsilon}}r^{-1}dr+O_\varepsilon(1)\\
&=\frac{\omega_4}{2}(|\ln \varepsilon|+\ln 2)+O_\varepsilon(1),
\end{aligned}
\end{equation}
 where $\omega_{4}$ is the area of unit sphere in ${\mathbb{R}^{4}}$ and $O_\varepsilon(1)$ is a constant associated with $\varepsilon$.
 We have that, for $N=3$
\begin{equation}\label{zz3}
\begin{aligned}
\int_{B_{2/\varepsilon}^{+}\setminus{B_{1}^{+}}}\frac{|y{'}|^{2}+y_{3}(y_{3}+x_{3}^{0})}
{{{{(1+|y{'}|^{2}+|y_{3}+x_{3}^{0}|^{2})^{2}}}}}dy
&=O\Big(\int_{B_{2/\varepsilon}^{+}\setminus{B_{1}^{+}}}\frac{1}{|y|^{2}}dy\Big)
=O(\varepsilon^{-1}).
\end{aligned}
\end{equation}
It follows from \eqref{z1}--\eqref{zz3} that
\begin{equation}\label{cz1}
\int_{\mathbb{R}^{N}_{+}}{\phi}^{2}{{U}_{\varepsilon}}(x\cdot{\nabla{{{U}_{\varepsilon}}}})dx
=
\begin{cases}
{-k_{4}^{2}\omega_4}{{\varepsilon}}^{2}|{\ln}\varepsilon|+O({{\varepsilon}}^{2}),~&~N=4,\\[1mm]
O({{\varepsilon}}),~&~N=3.
\end{cases}
\end{equation}

Arguing as above, we can calculate the last integral in the right-hand side of (\ref {3.05}) as follows
\begin{equation*}
  \begin{aligned}
\int_{\mathbb{R}^{N}_{+}}{\phi}^{2}|x|^{2}U_\varepsilon^{2}dx
&=\int_{B_{2}^{+}}|x|^{2}U_\varepsilon^{2}dx+\int_{ B_{2}^{+}\setminus{B_{1}^{+}}}({\phi}^{2}-1)|x|^{2}U_\varepsilon^{2}dx\\
&=\int_{B_{2}^{+}}|x|^{2}U_\varepsilon^{2}dx+O({\varepsilon}^{N-2})\\
&={\varepsilon}^{4}k_{N}^{2}\int_{{{B^{+}_{{2}/{\varepsilon}}}}}
\frac{|y|^{2}}{{({1+|y{'}|^{2}+|y_{N}+x_{N}^{0}|^{2})^{N-2}}}}dy+O({\varepsilon}^{N-2})\notag\\
&={\varepsilon}^{4}k_{N}^{2}\int_{B^{+}_{{2}/{\varepsilon}}\setminus{B_{1}^{+}}}
\frac{|y|^{2}}{{({1+|y{'}|^{2}+|y_{N}+x_{N}^{0}|^{2})^{N-2}}}}dy+O({\varepsilon}^{4})+O({\varepsilon}^{N-2}).\notag
  \end{aligned}
\end{equation*}
Moreover,
\begin{equation*}
  \begin{aligned}
\int_{{B^{+}_{{2}/{\varepsilon}}}\setminus{B_{1}^{+}}}
\frac{|y|^{2}}{{({1+|y{'}|^{2}+|y_{N}+x_{N}^{0}|^{2})^{N-2}}}}dy
&=O\Big(\int_{{B^{+}_{{2}/{\varepsilon}}}\setminus B_{1}^{+}}\frac{1}{|y|^{2N-6}}dy\Big)\notag\\
&=O\Big(\int_{1}^{{2}/{\varepsilon}}r^{5-N}dr\Big).\notag
  \end{aligned}
\end{equation*}
Hence, one has
\begin{equation}\label{cz3}
\int_{\mathbb{R}^{N}_{+}}{\phi}^{2}|x|^{2}U_\varepsilon^{2}dx=
\begin{cases}
O({\varepsilon}^{4}),~&~~N\geq 7,\\[1mm]
O({\varepsilon}^{4}|\ln{\varepsilon}|),~&~~N=6,\\[1mm]
O({\varepsilon}^{N-2}),~&~~3\leq N\leq5.\\[1mm]
\end{cases}
\end{equation}

We derive from \eqref{3.05}, \eqref{3.06}, \eqref{3.07}, \eqref{cz1} and \eqref{cz3} that
\begin{equation*}\label{13.09}
{\|\tilde{U}_{{\varepsilon}}\|}^{2}=
\begin{cases}
K_{1}+\frac{(N-2)k_{N}^{2}}{2}C_{1,N}\varepsilon^{2}+o({\varepsilon}^{2}),~&~~N\geq 5,\\[1mm]
K_{1}+\frac{k_{4}^{2}\omega_4}{2}{{\varepsilon}}^{2}|\ln{\varepsilon}|+O({\varepsilon}^{2}),~&~~N=4,\\[1mm]
K_{1}+O({\varepsilon}),~&~~N=3,\\[1mm]
\end{cases}
\end{equation*}
where $C_{1,N}$ is given by (\ref{adc}).
Therefore, we finish the proof.
\end{proof}

\begin{Lem}\label{lem3.3}
If $N\geq3$, we have
\begin{equation*}
{\|\tilde{U}_{{\varepsilon}}\|}_{L^{{2}^{*}}_{K}({\mathbb{R}}^{N}_{+})}^{2^{*}}=K_{2}-\beta_{N}{{\varepsilon}}^{2}+o({\varepsilon}^{2}),\notag
\end{equation*}
where $\varepsilon> 0$ is sufficiently small and
\begin{equation*}\label{4.02}
\beta_{N}=\frac{k_{N}^{2^{*}}}{2(N-2)}
\int_{\mathbb{R}^{N}_{+}}\frac{|y|^{2}}
{{{{(1+|y{'}|^{2}+|y_{N}+x_{N}^{0}|^{2})^{N}}}}}dy.
\end{equation*}
\end{Lem}
	
\begin{proof}
Note that
\begin{equation*}\label{w1}
	\begin{aligned}
{\|\tilde{U}_{{\varepsilon}}\|}_{L^{{2}^{*}}_{K}({\mathbb{R}}^{N}_{+})}^{2^{*}}
&=\int_{\mathbb{R}^{N}_{+}}K(x)\tilde{U}_{{\varepsilon}}^{2^{*}}dx=\int_{\mathbb{R}^{N}_{+}}K(x)^{\frac{2}{2-N}}{\phi}^{2^{*}}U_{\varepsilon}^{2^{*}}dx\\
&=\int_{\mathbb{R}^{N}_{+}}K(x)^{\frac{2}{2-N}}U_{\varepsilon}^{2^{*}}dx+O({\varepsilon}^{N})\\
&=\int_{\mathbb{R}^{N}_{+}}U_{\varepsilon}^{2^{*}}dx+\int_{\mathbb{R}^{N}_{+}}\big(K(x)^{\frac{2}{2-N}}-1\big)U_{\varepsilon}^{2^{*}}dx+O({\varepsilon}^{N}).
	\end{aligned}
\end{equation*}
By \eqref{2.111} and using the change of variables $y={x}/{{\varepsilon}}$, we conclude that
\begin{equation*}
  \begin{aligned}
{\|\tilde{U}_{{\varepsilon}}\|}_{L^{{2}^{*}}_{K}({\mathbb{R}}^{N}_{+})}^{2^{*}}
=K_{2}+k_{N}^{2^{*}}\int_{\mathbb{R}^{N}_{+}}\frac{K({\varepsilon}y)^{\frac{2}{2-N}}-1}
{{{{(1+|y{'}|^{2}+|y_{N}+x_{N}^{0}|^{2})^{N}}}}}dy+O({\varepsilon}^{N}).
  \end{aligned}
\end{equation*}
Since $K(x)= e^{|x|^{2}/4}$, there holds
\begin{equation*}
\int_{\mathbb{R}^{N}_{+}}\frac{K({\varepsilon}y)^{\frac{2}{2-N}}-1}
{{{{(1+|y{'}|^{2}+|y_{N}+x_{N}^{0}|^{2})^{N}}}}}dy
=\int_{\mathbb{R}^{N}_{+}}\frac{ e^{-\frac{{\varepsilon}^{2}|y|^{2}}{2(N-2)}}-1 }
{{{{(1+|y{'}|^{2}+|y_{N}+x_{N}^{0}|^{2})^{N}}}}}dy,
\end{equation*}
which yields that
\begin{equation}\label{w2}
  \begin{aligned}
{\|\tilde{U}_{{\varepsilon}}\|}_{L^{{2}^{*}}_{K}({\mathbb{R}}^{N}_{+})}^{2^{*}}
=K_{2}+k_{N}^{2^{*}}\int_{\mathbb{R}^{N}_{+}}\frac{ e^{-\frac{{\varepsilon}^{2}|y|^{2}}{2(N-2)}}-1 }
{{{{(1+|y{'}|^{2}+|y_{N}+x_{N}^{0}|^{2})^{N}}}}}dy+O({\varepsilon}^{N}).
  \end{aligned}
\end{equation}
Obviously,
\begin{equation}\label{w3}
  \begin{aligned}
&\ \ \ \ \int_{\mathbb{R}^{N}_{+}}\frac{ e^{-\frac{{\varepsilon}^{2}|y|^{2}}{2(N-2)}}-1 }
{{{{(1+|y{'}|^{2}+|y_{N}+x_{N}^{0}|^{2})^{N}}}}}dy
-\int_{\mathbb{R}^{N}_{+}}\frac{-\frac{\varepsilon^{2}|y|^{2}}{2(N-2)}}
{{{{(1+|y{'}|^{2}+|y_{N}+x_{N}^{0}|^{2})^{N}}}}}dy\\
&=\int_{B_{1/\varepsilon}^{+}}\frac{ e^{-\frac{{\varepsilon}^{2}|y|^{2}}{2(N-2)}}-1+\frac{\varepsilon^{2}|y|^{2}}{2(N-2)} }
{{{{(1+|y{'}|^{2}+|y_{N}+x_{N}^{0}|^{2})^{N}}}}}dy+\int_{\mathbb{R}^{N}_{+}\setminus B_{1/\varepsilon}^{+}}\frac{ e^{-\frac{{\varepsilon}^{2}|y|^{2}}{2(N-2)}}-1+\frac{\varepsilon^{2}|y|^{2}}{2(N-2)} }
{{{{(1+|y{'}|^{2}+|y_{N}+x_{N}^{0}|^{2})^{N}}}}}dy.
  \end{aligned}
\end{equation}
It follows from Taylor's formula that
\begin{equation*}
e^{-\frac{{\varepsilon}^{2}|y|^{2}}{2(N-2)}}-1=-\frac{{\varepsilon}^{2}|y|^{2}}{2(N-2)}+O({\varepsilon}^{4}|y|^4),\quad\quad
y\in B_{1/\varepsilon}^{+},
\end{equation*}
and then
\begin{equation*}
  \begin{aligned}
\int_{B_{1/\varepsilon}^{+}}\frac{ e^{-\frac{{\varepsilon}^{2}|y|^{2}}{2(N-2)}}-1+\frac{\varepsilon^{2}|y|^{2}}{2(N-2)} }
{{{{(1+|y{'}|^{2}+|y_{N}+x_{N}^{0}|^{2})^{N}}}}}dy
&=O\Big(\varepsilon^4\int_{B_{1/\varepsilon}^{+}}\frac{ |y|^{4}}
{{{{(1+|y{'}|^{2}+|y_{N}+x_{N}^{0}|^{2})^{N}}}}}dy\Big)\\
&=O\Big(\varepsilon^4\int_{B_{1/\varepsilon}^{+}\setminus B_{1}^+}\frac{ |y|^{4}}
{{{{(1+|y{'}|^{2}+|y_{N}+x_{N}^{0}|^{2})^{N}}}}}dy\Big)
+O(\varepsilon^4).
  \end{aligned}
  \end{equation*}
Due to
\begin{equation*}
  \begin{aligned}
\int_{B_{1/\varepsilon}^{+}\setminus B_{1}^+}\frac{ |y|^{4}}
{{{{(1+|y{'}|^{2}+|y_{N}+x_{N}^{0}|^{2})^{N}}}}}dy
&=O\Big(\int_{B_{1/\varepsilon}^{+}\setminus B_{1}^+}\frac{ 1}
{|y|^{2N-4}}dy\Big)\\
&=O\Big(\int_{1}^{{{1}/\varepsilon}}r^{-N+3}dr\Big),
  \end{aligned}
\end{equation*}
we obtain that for $N\geq 3$,
\begin{equation}\label{w5}
\int_{B_{1/\varepsilon}^{+}}\frac{ e^{-\frac{{\varepsilon}^{2}|y|^{2}}{2(N-2)}}-1+\frac{\varepsilon^{2}|y|^{2}}{2(N-2)} }
{{{{(1+|y{'}|^{2}+|y_{N}+x_{N}^{0}|^{2})^{N}}}}}dy=o(\varepsilon^2).
\end{equation}
On the other hand, one get
\begin{equation}\label{w6}
  \begin{aligned}
&\ \ \ \ \int_{\mathbb{R}^{N}_{+}\setminus B_{1/\varepsilon}^{+}}\frac{ e^{-\frac{{\varepsilon}^{2}|y|^{2}}{2(N-2)}}-1+\frac{\varepsilon^{2}|y|^{2}}{2(N-2)} }
{{{{(1+|y{'}|^{2}+|y_{N}+x_{N}^{0}|^{2})^{N}}}}}dy\\
&=O\Big(\int_{\mathbb{R}^{N}_{+}\setminus B_{1/\varepsilon}^{+}}\frac{ 1}
{|y|^{2N}}dy\Big)+
O\Big(\varepsilon^2\int_{\mathbb{R}^{N}_{+}\setminus B_{1/\varepsilon}^{+}}\frac{ 1}
{|y|^{2N-2}}dy\Big)\\
&=O\Big(\int_{{{1}/\varepsilon}}^{\infty}r^{-N-1}dr\Big)+O\Big(\varepsilon^2\int_{{{1}/\varepsilon}}^{\infty}r^{-N+1}dr\Big)
=O(\varepsilon^N).
  \end{aligned}
\end{equation}
Substituting \eqref{w5} and \eqref{w6} into \eqref{w3}, we have that
\begin{equation}\label{w7}
  \begin{aligned}
& \ \ \ \ \int_{\mathbb{R}^{N}_{+}}\frac{ e^{-\frac{{\varepsilon}^{2}|y|^{2}}{2(N-2)}}-1 }
{{{{(1+|y{'}|^{2}+|y_{N}+x_{N}^{0}|^{2})^{N}}}}}dy\\
&=-\frac{\varepsilon^{2}}{2(N-2)}\int_{\mathbb{R}^{N}_{+}}\frac{|y|^{2}}
{{{{(1+|y{'}|^{2}+|y_{N}+x_{N}^{0}|^{2})^{N}}}}}dy+o({\varepsilon}^{2}).
  \end{aligned}
\end{equation}
It follows from \eqref{w2} and \eqref{w7} that
\begin{equation*}\label{3.016}
{\|\tilde{U}_{{\varepsilon}}\|}_{{L}_{K}^{2^{*}}({\mathbb{R}}^{N}_{+})}^{2^{*}}
=K_{2}- {\varepsilon}^{2}\frac{k_{N}^{2^{*}}}{2(N-2)}\int_{\mathbb{R}^{N}_{+}}\frac{|y|^{2}}
{{{{(1+|y{'}|^{2}+|y_{N}+x_{N}^{0}|^{2})^{N}}}}}dy+o({\varepsilon}^{2}).
\end{equation*}
We complete the proof of this Lemma.
\end{proof}

\begin{Lem}\label{lem3.4}
If $N\geq3$, one has

\begin{equation*}
{\|\tilde{U}_{{\varepsilon}}\|}_{L^{{2}_{*}}_{K}({\mathbb{R}}^{N-1})}^{2_{*}}
=
\begin{cases}
K_{3}-\gamma_{N}{{\varepsilon}}^{2}+o({\varepsilon}^{2}),~&~~N\geq4,\\[1mm]
K_{3}+O({\varepsilon}^2|\ln\varepsilon|),~&~~N=3,\\[1mm]
\end{cases}
\end{equation*}
where $\varepsilon> 0$ is sufficiently small and
\begin{equation*}\label{4.03}
\gamma_{N}=\frac{k_{N}^{2_{*}}}{4(N-2)}\int_{\mathbb{R}^{N-1}}\frac{|y{'}|^{2}}
{{{{(1+|y{'}|^{2}+|x_{N}^{0}|^{2})^{N-1}}}}}dy{'}.
\end{equation*}
\end{Lem}

\begin{proof}
 For $N\geq4$, we have that
\begin{equation}\label{xx1}
  \begin{aligned}
{\|\tilde{U}_{{\varepsilon}}\|}_{{L}_{K}^{2_{*}}({\mathbb{R}}^{N-1})}^{2_{*}}
&=\int_{\mathbb{R}^{N-1}}K(x{'},0)^{\frac{1}{2-N}}U^{2_{*}}_{\varepsilon}dx{'}+O({\varepsilon}^{N-1})\\
&=K_{3}+k_{N}^{2_{*}}\int_{\mathbb{R}^{N-1}}\frac{e^{-\frac{{\varepsilon}^{2}|y'|^{2}}{4(N-2)}}-1}
{{{{(1+|y{'}|^{2}+|x_{N}^{0}|^{2})^{N-1}}}}}dy{'}+O({\varepsilon}^{N-1}).
  \end{aligned}
\end{equation}
It is clear that
\begin{equation}\label{xx2}
  \begin{aligned}
& \ \ \ \ \int_{\mathbb{R}^{N-1}}\frac{ e^{-\frac{{\varepsilon}^{2}|y{'}|^{2}}{4(N-2)}}-1 }
{{{{(1+|y{'}|^{2}+|x_{N}^{0}|^{2})^{N-1}}}}}dy{'}
-\int_{\mathbb{R}^{N-1}}\frac{-\frac{\varepsilon^{2}|y{'}|^{2}}{4(N-2)}}
{{{{(1+|y{'}|^{2}+|x_{N}^{0}|^{2})^{N-1}}}}}dy{'}\\
&=\int_{\hat{B}_{1/\varepsilon}}\frac{ e^{-\frac{{\varepsilon}^{2}|y{'}|^{2}}{4(N-2)}}-1+\frac{\varepsilon^{2}|y{'}|^{2}}{4(N-2)} }
{{{{(1+|y{'}|^{2}+|x_{N}^{0}|^{2})^{N-1}}}}}dy{'}+\int_{\mathbb{R}^{N-1}\setminus \hat{B}_{1/\varepsilon}}\frac{ e^{-\frac{{\varepsilon}^{2}|y{'}|^{2}}{4(N-2)}}-1+\frac{\varepsilon^{2}|y{'}|^{2}}{4(N-2)} }
{{{{(1+|y{'}|^{2}+|x_{N}^{0}|^{2})^{N-1}}}}}dy{'},
  \end{aligned}
\end{equation}
where $\hat{B}_{r}=\hat{B}_{r}(0)\subset {\mathbb{R}}^{N-1}$ for any $r>0$ is a ball. From Taylor's formula, we have
\begin{equation}\label{t1}
e^{-\frac{{\varepsilon}^{2}|y{'}|^{2}}{4(N-2)}}-1=-\frac{{\varepsilon}^{2}|y{'}|^{2}}{4(N-2)}+O({\varepsilon}^{4}|y{'}|^4),\quad\quad
y{'}\in \hat{B}_{1/\varepsilon},
\end{equation}
and then
\begin{equation*}
  \begin{aligned}
\int_{\hat{B}_{1/\varepsilon}}\frac{ e^{-\frac{{\varepsilon}^{2}|y{'}|^{2}}{4(N-2)}}-1+\frac{\varepsilon^{2}|y{'}|^{2}}{4(N-2)} }
{{{{(1+|y{'}|^{2}+|x_{N}^{0}|^{2})^{N-1}}}}}dy{'}
&=O\Big(\varepsilon^4\int_{\hat{B}_{1/\varepsilon}}\frac{ |y{'}|^{4}}
{{{{(1+|y{'}|^{2}+|x_{N}^{0}|^{2})^{N-1}}}}}dy{'}\Big)\\
&=O\Big(\varepsilon^4\int_{\hat{B}_{1/\varepsilon}\setminus \hat{B}_{1}}\frac{ |y{'}|^{4}}
{{{{(1+|y{'}|^{2}+|x_{N}^{0}|^{2})^{N-1}}}}}dy{'}\Big)
+O(\varepsilon^4).
  \end{aligned}
  \end{equation*}
Since
\begin{equation*}
  \begin{aligned}
\int_{\hat{B}_{1/\varepsilon}\setminus \hat{B}_{1}}\frac{ |y{'}|^{4}}
{{{{(1+|y{'}|^{2}+|x_{N}^{0}|^{2})^{N-1}}}}}dy{'}
=O\Big(\int_{\hat{B}_{1/\varepsilon}\setminus \hat{B}_{1}}\frac{ 1}
{|y{'}|^{2N-6}}dy{'}\Big)
=O\Big(\int_{1}^{{{1}/\varepsilon}}r^{-N+4}dr\Big),
  \end{aligned}
\end{equation*}
one has
\begin{equation}\label{xx4}
\int_{\hat{B}_{1/\varepsilon}}\frac{ e^{-\frac{{\varepsilon}^{2}|y{'}|^{2}}{4(N-2)}}-1+\frac{\varepsilon^{2}|y{'}|^{2}}{4(N-2)} }
{{{{(1+|y{'}|^{2}+|x_{N}^{0}|^{2})^{N-1}}}}}dy{'}=o(\varepsilon^2).
\end{equation}
Moreover, there holds
\begin{equation}\label{xx5}
  \begin{aligned}
& \ \ \ \ \int_{\mathbb{R}^{N-1}\setminus \hat{B}_{1/\varepsilon}}\frac{ e^{-\frac{{\varepsilon}^{2}|y{'}|^{2}}{4(N-2)}}-1+\frac{\varepsilon^{2}|y{'}|^{2}}{4(N-2)} }
{{{{(1+|y{'}|^{2}+|x_{N}^{0}|^{2})^{N-1}}}}}dy{'}\\
&=O\Big(\int_{\mathbb{R}^{N-1}\setminus \hat{B}_{1/\varepsilon}}\frac{ 1}
{|y{'}|^{2N-2}}dy{'}\Big)+
O\Big(\varepsilon^2\int_{\mathbb{R}^{N-1}\setminus \hat{B}_{1/\varepsilon}}\frac{ 1}
{|y{'}|^{2N-4}}dy{'}\Big)\\
&=O\Big(\int_{{{1}/\varepsilon}}^{\infty}r^{-N}dr\Big)+O\Big(\varepsilon^2\int_{{{1}/\varepsilon}}^{\infty}r^{-N+2}dr\Big)\\
&=O(\varepsilon^{N-1}).
  \end{aligned}
\end{equation}
We conclude from \eqref{xx2}, \eqref{xx4} and \eqref{xx5} that for $N\geq4$,
\begin{equation}\label{xx7}
  \begin{aligned}
& \ \ \ \ \int_{\mathbb{R}^{N-1}}\frac{e^{-\frac{{\varepsilon}^{2}|y'|^{2}}{4(N-2)}}-1}
{{{{(1+|y{'}|^{2}+|x_{N}^{0}|^{2})^{N-1}}}}}dy{'}\\
&=-\frac{\varepsilon^{2}}{4(N-2)}\int_{\mathbb{R}^{N-1}}\frac{|y{'}|^{2}}
{{{{(1+|y{'}|^{2}+|x_{N}^{0}|^{2})^{N-1}}}}}dy{'}+o({\varepsilon}^{2}).
  \end{aligned}
\end{equation}
Combining \eqref{xx1} with \eqref{xx7} we have  that
\begin{equation}\label{xx8}
{\|\tilde{U}_{{\varepsilon}}\|}_{{L}_{K}^{2_{*}}({\mathbb{R}}^{N-1})}^{2_{*}}
=K_{3}- {\varepsilon}^{2}\frac{k_{N}^{2_{*}}}{4(N-2)}\int_{\mathbb{R}^{N-1}}\frac{|y{'}|^{2}}
{{{{(1+|y{'}|^{2}+|x_{N}^{0}|^{2})^{N-1}}}}}dy{'}+o({\varepsilon}^{2}).
\end{equation}

For $N=3$, we conclude that
\begin{equation}\label{q1}
  \begin{aligned}
{\|\tilde{U}_{{\varepsilon}}\|}_{{L}_{K}^{4}({\mathbb{R}}^{2})}^{4}
&=\int_{\mathbb{R}^{2}}K(x{'},0)^{{-1}}{\phi}^{4}U^{4}_{\varepsilon}dx{'}\\
&=\int_{\hat{B}_{2}}K(x{'},0)^{{-1}}U^{4}_{\varepsilon}dx{'}+
\int_{\hat{B}_{2}\setminus{\hat{B}_{1}}}K(x{'},0)^{{-1}}\big({\phi}^{4}-1\big)U^{4}_{\varepsilon}dx{'}\\
&=\int_{\hat{B}_{2}}K(x{'},0)^{{-1}}U^{4}_{\varepsilon}dx{'}+O({\varepsilon}^2)\\
&=\int_{\hat{B}_{2}}U^{4}_{\varepsilon}dx{'}
+\int_{\hat{B}_{2}}\big(K(x{'},0)^{{-1}}-1\big)U^{4}_{\varepsilon}dx{'}+O({\varepsilon}^2)\\
&=K_{3}+k_{3}^{4}\int_{\hat{B}_{{2}/{\varepsilon}}}\frac{e^{-\frac{{\varepsilon}^{2}|y'|^{2}}{4}}-1}
{{{{(|y{'}|^{2}+4)^{2}}}}}dy{'}+O({\varepsilon}^{2}).
  \end{aligned}
\end{equation}
Similar as (\ref{t1}), applying  Taylor's formula, we obtain
\begin{equation}\label{q02}
  \begin{aligned}
\int_{\hat{B}_{2/\varepsilon}}\frac{ e^{-\frac{{\varepsilon}^{2}|y{'}|^{2}}{4}}-1}
{{{{(|y{'}|^{2}+4)^{2}}}}}dy{'}
&=-\frac{\varepsilon^{2}}{4}\int_{\hat{B}_{2/\varepsilon}}\frac{|y{'}|^{2}}
{(|y{'}|^{2}+4)^{2}}dy{'}
+O\Big(\varepsilon^4\int_{\hat{B}_{2/\varepsilon}}\frac{ |y{'}|^{4}}
{(|y{'}|^{2}+4)^{2}}dy{'}\Big)\\
&=-c_1\varepsilon^{2}-\frac{\varepsilon^{2}}{4}\int_{\hat{B}_{2/\varepsilon}\setminus \hat{B}_{1}}\frac{|y{'}|^{2}}
{(|y{'}|^{2}+4)^{2}}dy{'}+O(\varepsilon^4)
\\
&\quad+O\Big(\varepsilon^4\int_{\hat{B}_{2/\varepsilon}\setminus \hat{B}_{1}}\frac{ |y{'}|^{4}}
{(|y{'}|^{2}+4)^{2}}dy{'}\Big),
  \end{aligned}
  \end{equation}
where $c_1$ is a positive constant. Note that
\begin{equation*}
\int_{\hat{B}_{2/\varepsilon}\setminus \hat{B}_{1}}\frac{ |y{'}|^{2}}
{(|y{'}|^{2}+4)^{2}}dy{'}
=O\Big(\int_{\hat{B}_{2/\varepsilon}\setminus \hat{B}_{1}}\frac{ 1}
{|y{'}|^{2}}dy{'}\Big)
=O\Big(\int_{1}^{{{2}/\varepsilon}}r^{-1}dr\Big)
=O(|\ln\varepsilon|)
\end{equation*}
and
\begin{equation*}
\int_{\hat{B}_{2/\varepsilon}\setminus \hat{B}_{1}}\frac{ |y{'}|^{4}}
{(|y{'}|^{2}+4)^{2}}dy{'}
=O\Big(\int_{\hat{B}_{2/\varepsilon}\setminus \hat{B}_{1}}dy{'}\Big)
=O(\varepsilon^{-2}).
\end{equation*}
It follows from \eqref{q02} that
\begin{equation}\label{q5}
\int_{\hat{B}_{2/\varepsilon}}\frac{ e^{-\frac{{\varepsilon}^{2}|y{'}|^{2}}{4}}-1}
{{{{(|y{'}|^{2}+4)^{2}}}}}dy{'}=O(\varepsilon^2|\ln\varepsilon|).
\end{equation}
We deduce from \eqref{q1} and \eqref{q5} that for $N=3$,
\begin{equation}\label{q4}
{\|\tilde{U}_{{\varepsilon}}\|}_{L^{{2}_{*}}_{K}({\mathbb{R}}^{N-1})}^{2_{*}}
=K_{3}+O({\varepsilon}^2|\ln\varepsilon|).
\end{equation}
Hence, the proof is completed from \eqref{xx8} and \eqref{q4}.
\end{proof}

\begin{Lem}\label{lem3.6}
Let $N\geq3$, $p\in (2, 2^{*})$ and $\theta _{N}:=N-\frac{(N-2)p}{2}$. Then we have 
\begin{equation}\label{le0}
{\|\tilde{U}_{{\varepsilon}}\|}_{L^{p}_{K}({\mathbb{R}}^{N})}^{p}
\geq
\begin{cases}
b_1{\varepsilon}^{\theta_N}+o({\varepsilon}^{\theta_N}),~&~~N\geq4,~2<p<2^*,\\[1mm]
b_2{\varepsilon}^{3-\frac{p}{2}}+o({\varepsilon}^{3-\frac{p}{2}}),~&~~N=3,~3<p<6,\\[1mm]
b_3{\varepsilon}^{\frac{3}{2}}|\ln{\varepsilon}|+O({\varepsilon}^{\frac{3}{2}}),~&~~N=3,~p=3,\\[1mm]
b_4{\varepsilon}^{\frac{p}{2}}+o({\varepsilon}^{\frac{p}{2}}),~&~~N=3,~2<p<3,\\[1mm]
\end{cases}
\end{equation}
where $\varepsilon>0$ is sufficiently small, $b_{1}, b_2, b_3, b_4$ are positive constants independent of $\varepsilon$.
\end{Lem}

\begin{proof}
Since $p\in (2, 2^{*})$, one has $K(x)^{1-\frac{p}{2}}\geq e^{-\frac{p-2}{2}}>0$ for each $|x|\leq 2$.
From the definition of ${\phi}$, we calculate
\begin{equation*}
\begin{aligned}
{\|\tilde{U}_{{\varepsilon}}\|}_{L^{p}_{K}({\mathbb{R}}^{N}_{+})}^{p}
&=\int_{\mathbb{R}^{N}_+}K(x)\tilde{U}_{{\varepsilon}}^{p}dx
=\int_{\mathbb{R}^{N}_+}\frac{K(x)^{1-\frac{p}{2}}\phi(x)^{p}k_{N}^{p}\varepsilon^{\frac{(N-2)p}{2}}}{(\varepsilon^{2}+|x'|^{2}+|x_{N}+\varepsilon x_{N}^{0}|^{2})^{\frac{(N-2)p}{2}}}
dx\\
&\geq e^{-\frac{p-2}{2}} \int_{B_{2}^{+}} \frac{\phi(x)^{p}k_{N}^{p}\varepsilon^{\frac{(N-2)p}{2}}}{(\varepsilon^{2}+|x'|^{2}+|x_{N}+\varepsilon x_{N}^{0}|^{2})^{\frac{(N-2)p}{2}}} dx \\
&\geq e^{-\frac{p-2}{2}} k_{N}^{p}\varepsilon^{\frac{(N-2)p}{2}}\int_{B_{1}^{+}} \frac{1}{(\varepsilon^{2}+|x'|^{2}+|x_{N}+\varepsilon x_{N}^{0}|^{2})^{\frac{(N-2)p}{2}}} dx\\
&=e^{-\frac{p-2}{2}} k_{N}^{p}\varepsilon^{\theta _{N}}\int_{B_{1/\varepsilon}^{+}} \frac{1}{(1+|y'|^{2}+|y_{N}+ x_{N}^{0}|^{2})^{\frac{(N-2)p}{2}}} dy\\
&=e^{-\frac{p-2}{2}} k_{N}^{p}\varepsilon^{\theta _{N}}\int_{B_{1}^{+}} \frac{1}{(1+|y'|^{2}+|y_{N}+ x_{N}^{0}|^{2})^{\frac{(N-2)p}{2}}} dy\\
&\ \ \ \ +e^{-\frac{p-2}{2}} k_{N}^{p}\varepsilon^{\theta _{N}}\int_{B_{1/\varepsilon}^{+} \setminus B_{1}^{+}} \frac{1}{(1+|y'|^{2}+|y_{N}+ x_{N}^{0}|^{2})^{\frac{(N-2)p}{2}}} dy,
\end{aligned}
\end{equation*}	
where $\theta _{N}=N-\frac{(N-2)p}{2}$.
Note that there exists a positive constant $d>0$ such that
\begin{equation*}
\frac{1}{(1+|y'|^{2}+|y_{N}+ x_{N}^{0}|^{2})^{\frac{(N-2)p}{2}}}\geq
\frac{d}{|y|^{(N-2)p}}, \ \ y\in B_{1/\varepsilon}^{+} \setminus B_{1}^{+}.
\end{equation*}
If $p=\frac{N}{N-2}$, there holds
\begin{equation}\label{le1}
	\begin{aligned}
	{\|\tilde{U}_{{\varepsilon}}\|}_{L^{p}_{K}({\mathbb{R}}^{N}_{+})}^{p}
	&\geq d_{1}\varepsilon^{\theta_{N}}+
	d_{2}\varepsilon^{\theta_{N}}\int_{B_{1/\varepsilon}^{+} \setminus B_{1}^{+}} \frac{1}{(1+|y'|^{2}+|y_{N}+ x_{N}^{0}|^{2})^{\frac{(N-2)p}{2}}} dy\\
	&\geq d_{1}\varepsilon^{\theta_{N}}+
	d_{2}\varepsilon^{\theta_{N}}d \int_{B_{1/\varepsilon}^{+} \setminus B_{1}^{+}} \frac{1}{|y|^{(N-2)p}}dy\\
	&= d_{1}\varepsilon^{\theta_{N}}+
	d_{2}\varepsilon^{\theta_{N}}\frac{d \omega_N}{2} \int_{1}^{{{1}/\varepsilon}}r^{-1}dr
	= d_{1}\varepsilon^{\frac{N}{2}}+ \bar{d}_{2} \varepsilon^{\frac{N}{2}}|\ln \varepsilon|,
	\end{aligned}
\end{equation}
where $d_{1}$, $d_{2}$, $\bar{d}_{2}>0$. If $p<\frac{N}{N-2}$, we have
\begin{equation}\label{le2}
	\begin{aligned}
		{\|\tilde{U}_{{\varepsilon}}\|}_{L^{p}_{K}({\mathbb{R}}^{N}_{+})}^{p}
		&\geq d_{1}\varepsilon^{\theta_{N}}+
		d_{2}\varepsilon^{\theta_{N}}d \int_{B_{1/\varepsilon}^{+} \setminus B_{1}^{+}} \frac{1}{|y|^{(N-2)p}}dy\\
		&= d_{1}\varepsilon^{\theta_{N}}+
		d_{2}\varepsilon^{\theta_{N}}\frac{d \omega_N}{2} \int_{1}^{{{1}/\varepsilon}}r^{N-1-(N-2)p}dr\\
		&\geq d_{1}\varepsilon^{\theta_{N}}+ \frac{d d_{2} \omega_N}{2(N-(N-2)p)} \varepsilon^{\theta_{N}}(\varepsilon^{-N+(N-2)p} -1)\\
		&= d_{3}\varepsilon^{\frac{(N-2)p}{2}}+o(\varepsilon^{\frac{(N-2)p}{2}}),
	\end{aligned}
\end{equation}
where $d_{3}>0$. Similarly, if $p>\frac{N}{N-2}$, we obtain
\begin{equation}\label{le3}
	\begin{aligned}
		{\|\tilde{U}_{{\varepsilon}}\|}_{L^{p}_{K}({\mathbb{R}}^{N}_{+})}^{p}
		&= d_{1}\varepsilon^{\theta_{N}}+
		d_{2}\varepsilon^{\theta_{N}}\frac{d \omega_N}{2} \int_{1}^{{{1}/\varepsilon}}r^{N-1-(N-2)p}dr\\
		&\geq d_{1}\varepsilon^{\theta_{N}}+ \frac{d d_{2} \omega_N}{2((N-2)p-N)} \varepsilon^{\theta_{N}}(1-\varepsilon^{-N+(N-2)p})\\
		&= d_{4}\varepsilon^{\theta_{N}}+o(\varepsilon^{\theta_{N}}),
	\end{aligned}
\end{equation}
where $d_{4}>0$.
From \eqref{le1}, \eqref{le2} and \eqref{le3}, one has
\begin{equation*}
{\|\tilde{U}_{{\varepsilon}}\|}_{L^{p}_{K}({\mathbb{R}}^{N})}^{p}
\geq
\begin{cases}
d_{3}\varepsilon^{\frac{(N-2)p}{2}}+o(\varepsilon^{\frac{(N-2)p}{2}}),~&~~p<\frac{N}{N-2},\\[1mm]
\bar{d}_{2} \varepsilon^{\frac{N}{2}}|\ln \varepsilon| +d_{1}\varepsilon^{\frac{N}{2}} ,~&~~p=\frac{N}{N-2},\\[1mm]
d_{4}\varepsilon^{\theta_{N}}+o(\varepsilon^{\theta_{N}}), ~&~~p>\frac{N}{N-2},
\end{cases}
\end{equation*}
which implies that \eqref{le0} holds.
\end{proof}

Denote
\begin{align*}
K_{1}({\varepsilon}):&={\|\tilde{U}_{{\varepsilon}}\|}^2,\\
K_{2}({\varepsilon}):&={\|\tilde{U}_{{\varepsilon}}\|}_{L^{{2}^{*}}_{K}({\mathbb{R}}^{N}_{+})}^{2^{*}},\\
K_{3}({\varepsilon}):&={\|\tilde{U}_{{\varepsilon}}\|}_{L^{{2}_{*}}_{K}({\mathbb{R}}^{N-1})}^{2_{*}},\\
K_{4}({\varepsilon}):&={\|\tilde{U}_{{\varepsilon}}\|}_{L^{p}_{K}({\mathbb{R}}^{N})}^{p}.
\end{align*}
Now, we are ready to prove \eqref{0}.

\begin{Lem}\label{lem3.8}  For any fixed $\lambda >0$, the inequality (\ref {0}) holds, and (\ref {2.8}) is naturally obtained,  if one of the following assumptions holds:
\vskip 0.2cm
	
	$( \romannumeral 1)$  $~N\geq4, 2<p<2^* \ \text{and}~ \mu>0$;
\vskip 0.2cm
	
	$( \romannumeral 2)$  $~N=3, 4<p<6 \
\text{and}~ \mu>0$;
	\vskip 0.2cm
	
	$( \romannumeral 3)$   $~N=3, 2<p\leq4 \ \text{and}~$ $\mu>0$~\text{sufficiently large}.
\end{Lem}
	
\begin{proof}
Define the function
\begin{equation*}
{g}_{{\varepsilon}}(t):=\frac{K_{1}({\varepsilon})}{2} t^{2}-\frac{\mu{\lambda^{\frac{(N-2)(2-p)}{4}}}K_{4}({\varepsilon})}{p} t^{p}
-\frac{K_{2}({\varepsilon})}{2^{*}}t^{2^{*}}-\frac{K_{3}({\varepsilon})}{2_{*}}t^{2_{*}},~t>0.\notag
\end{equation*}
The inequality (\ref {0}) holds if we verify that
\begin{equation}\label{3.0025}
\sup\limits_{t>0}g_{{\varepsilon}}(t)<A.
\end{equation}
Let $\tilde{t_{\varepsilon}}>0$ be a constant such that ${g}_{{\varepsilon}}(t)$ attains its maximum. One has
\begin{equation*}
 K_{1}({\varepsilon})-\mu\lambda^{\frac{(N-2)(2-p)}{4}}
  K_{4}({\varepsilon})\tilde{t_{\varepsilon}}^{{p-2}}-
K_{2}({\varepsilon})\tilde{t_{\varepsilon}}^{{2^{*}-2}}-K_{3}({\varepsilon})\tilde{t_{\varepsilon}}^{{2_{*}}-2}=0.\notag
\end{equation*}
From Lemmas \ref{lem3.1}-\ref{lem3.6} and $K_{1}=K_{2}+K_{3}$, we have $\tilde{t_{\varepsilon}}\rightarrow 1$ as ${\varepsilon}\rightarrow 0$, which yields that there exists $a_{1}>0$, independent of ${\varepsilon}$, such that $\tilde{t_{\varepsilon}}\geq a_{1}$ for any ${\varepsilon}>0$ small enough. Therefore,
\begin{equation}\label{v1}
  \begin{aligned}
g_{{\varepsilon}}(\tilde{t_{\varepsilon}})
&\leq\sup\limits_{t>0}\Big( \frac{K_{1}({\varepsilon})}{2} t^{2}-
\frac{K_{2}({\varepsilon})}{2^{*}}t^{2^{*}}-\frac{K_{3}({\varepsilon})}{2_{*}}t^{2_{*}}\Big)-\frac{\mu{\lambda^{\frac{(N-2)(2-p)}{4}}      }}{p} K_{4}({\varepsilon})\tilde{t_{\varepsilon}}^{p}\\
&\leq\sup\limits_{t>0}\Big(\frac{K_{1}({\varepsilon})}{2} t^{2}-\frac{K_{2}({\varepsilon})}{2^{*}}t^{2^{*}}-\frac{K_{3}({\varepsilon})}{2_{*}}t^{2_{*}}\Big)-\frac{\mu{\lambda^{\frac{(N-2)(2-p)}{4}}      }a_{1}^{p}}{p}  K_{4}({\varepsilon}).
  \end{aligned}
\end{equation}
Let $t_{{\varepsilon}}$ be the positive constant such that
\begin{equation*}
f_{{\varepsilon}}(t_\varepsilon)
=\sup\limits_{t>0}f_{{\varepsilon}}(t),
\end{equation*}
where \begin{equation*}
f_{{\varepsilon}}(t)
:=\frac{K_{1}({\varepsilon})}{2} t^{2}-\frac{K_{2}({\varepsilon})}{2^{*}}t^{2^{*}}-\frac{K_{3}({\varepsilon})}{2_{*}}t^{2_{*}}.
\end{equation*}
Thus $t_{{\varepsilon}}$ satisfies
\begin{equation*}
K_{1}({\varepsilon})-
K_{2}({\varepsilon})t_\varepsilon^{{2^{*}-2}}-K_{3}({\varepsilon})t_\varepsilon^{{2_{*}}-2}=0.\notag
\end{equation*}
Noting that $2^{*}-2=2(2_{*}-2)$, we have
\begin{equation*}
t_{{\varepsilon}}^{2_{*}-2}=\frac{-K_{3}({\varepsilon})+\sqrt{K_{3}^{2}({\varepsilon})
+4K_{2}({\varepsilon})K_{1}({\varepsilon})}}{2K_{2}({\varepsilon})}.\notag
\end{equation*}

If $N\geq5$, we deduce from Lemmas \ref{lem3.1}-\ref{lem3.6} and $K_{1}=K_{2}+K_{3}$ that
\begin{equation*}
t_{{\varepsilon}}^{2_{*}-2}=\frac{2K_{2}+O({{\varepsilon}}^{2})}{2K_{2}+O({{\varepsilon}}^{2})}=1+O({{\varepsilon}}^{2}),\notag
\end{equation*}
which means that $\Delta{t_{{\varepsilon}}}:={t_{{\varepsilon}}}-1=O({{\varepsilon}}^{2})$. From Taylor's formula, one has that for any $s>1$,
\begin{equation*}\label{03}
t_{{\varepsilon}}^{s}=1+s\Delta{t_{{\varepsilon}}}+O({{\varepsilon}}^{4}).
\end{equation*}
Combining the above with \eqref{2000},
we obtain that
\begin{equation*}
  \begin{aligned}
f_{{\varepsilon}}({t}_{{\varepsilon}})
&=\frac{K_{1}({\varepsilon})}{2}{t}_{{\varepsilon}}^{2}
-\frac{K_{2}({\varepsilon})}{2^{*}}t_{{\varepsilon}}^{2^{*}}-
\frac{K_{3}({\varepsilon})}{2_{*}}t_{{\varepsilon}}^{2_{*}}\\
&=\frac{1}{2}(K_{1}+\alpha_{N}{{\varepsilon}}^{2})t_{{\varepsilon}}^{2}
-\frac{1}{2^{*}}(K_{2}-\beta_{N}{{\varepsilon}}^{2})t_{{\varepsilon}}^{2^{*}}-\frac{1}{2_{*}}(K_{3}-\gamma_{N}{{\varepsilon}}^{2})t_{{\varepsilon}}^{2_{*}}+
o({{\varepsilon}}^{2})\\
&=\frac{1}{2}(K_{1}+\alpha_{N}{{\varepsilon}}^{2})-\frac{1}{2^{*}}(K_{2}-\beta_{N}{{\varepsilon}}^{2})
-\frac{1}{2_{*}}(K_{3}-\gamma_{N}{{\varepsilon}}^{2})\\
& \ \ \ \ +(K_{1}-K_{2}-K_{3})\Delta{t_{{\varepsilon}}}+o({{\varepsilon}}^{2})\\
&=A+\Big(\frac{\alpha_{N}}{2}+\frac{\beta_{N}}{2^{*}}+\frac{\gamma_{N}}{2_{*}}\Big){{\varepsilon}}^{2}+o({{\varepsilon}}^{2}).
  \end{aligned}
\end{equation*}
That is, for $N\geq5$,
\begin{equation}\label{za1}
\sup\limits_{t>0}f_{{\varepsilon}}({t})=
A+O({\varepsilon}^{2})+o({{\varepsilon}}^{2}).
\end{equation}

If $N=4$, one has
\begin{equation*}
t_{{\varepsilon}}^{2_{*}-2}
=1+O({{\varepsilon}}^{2}|\ln{\varepsilon}|),\notag
\end{equation*}
which means that
\begin{equation}\label{za2}
\sup\limits_{t>0}f_{{\varepsilon}}({t})=
A+O({{\varepsilon}}^{2}|\ln{\varepsilon}|).
\end{equation}
In view of \eqref{za1} and \eqref{za2}, we have that for $N\geq 4$,
\begin{equation}\label{v2}
\sup\limits_{t>0}f_{{\varepsilon}}({t})=
A+B_\varepsilon,
\end{equation}
where
\begin{equation*}	
B_\varepsilon=
	\begin{cases}
		O({\varepsilon}^{2}),~&~N\geq5 ,\\
O({{\varepsilon}}^{2}|{\ln}\varepsilon|),~&~N=4 .
	\end{cases}
\end{equation*}

Similarly, we get that for $N=3$,
\begin{equation}\label{v4}
\sup\limits_{t>0}f_{{\varepsilon}}({t})=
A+O({{\varepsilon}}).
\end{equation}

Now, we are ready to verify (\ref {3.0025}).
For $N\ge 4$ and $2<p<2^{*}$, it follows from Lemma \ref{lem3.6}, \eqref{v1} and \eqref{v2} that for  $\varepsilon$ sufficiently small,
\begin{equation}\label{v3}
g_{{\varepsilon}}(\tilde{t_{\varepsilon}})\leq  
A-\frac{\mu{\lambda^{\frac{(N-2)(2-p)}{4}}      } a_{1}^{p}b_{1}}{p}{{\varepsilon}}^{\theta _{N}}+o({{\varepsilon}}^{\theta _{N}})<A,
\end{equation}
because $\theta _{N}\in (0,2)$, $\lambda>0$ and $\mu>0$.
 For $N=3$, from Lemma \ref{lem3.6}, \eqref{v1} and \eqref{v4}, we proceed as follows:
   \begin{enumerate}
     \item If $4<p<6$, we conclude  that for any $\lambda, \mu>0$, $\varepsilon>0$ sufficiently small,
\begin{equation}\label{v8}
g_{{\varepsilon}}(\tilde{t_{\varepsilon}})\leq  A-\frac{\mu{\lambda^{\frac{(N-2)(2-p)}{4}}      } a_{1}^{p}b_{2}}{p}{{\varepsilon}}^{3-\frac{p}{2}}+o({{\varepsilon}}^{3-\frac{p}{2}})<A,
\end{equation}
since $3-{\frac{p}{2}}\in(0, 1)$.
     \item If $3<p\leq4$ and $\lambda>0$, we can take $\mu=\varepsilon^{-\frac{1}{2}}$ such that
\begin{equation}\label{v5}
g_{{\varepsilon}}(\tilde{t_{\varepsilon}})\leq  
A-\frac{{\lambda^{\frac{(N-2)(2-p)}{4}}      } a_{1}^{p}b_{2}}{p}{{\varepsilon}}^{\frac{5-p}{2}}+O({\varepsilon})<A,
\end{equation}
for small $\varepsilon>0$ since ${\frac{5-p}{2}}\in[\frac{1}{2}, 1)$.
\item  If $p=3$ and $\lambda>0$, by taking $\mu=\varepsilon^{-\frac 12}$ and $\varepsilon>0$ small enough, we get that
\begin{equation}\label{v6}
g_{{\varepsilon}}(\tilde{t_{\varepsilon}})\leq  
A-\frac{{\lambda^{\frac{(N-2)(2-p)}{4}}      } a_{1}^{p}b_{3}}{p}{{\varepsilon}}|\ln \varepsilon|+O({\varepsilon})<A.
\end{equation}

 \item   If $2<p<3$ and $\lambda>0$, by taking $\mu=\varepsilon^{-\frac{1}{2}}$ and $\varepsilon>0$ sufficiently small, we conclude
\begin{equation}\label{v7}
g_{{\varepsilon}}(\tilde{t_{\varepsilon}})\leq  
A-\frac{{\lambda^{\frac{(N-2)(2-p)}{4}}      } a_{1}^{p}b_{4}}{p}{{\varepsilon}}^{\frac{p-1}{2}}+O({\varepsilon})<A,
\end{equation}
since $\frac{p-1}{2}\in (\frac{1}{2}, 1)$.
\end{enumerate}

Therefore, we conclude that \eqref{3.0025} holds for $\varepsilon>0$ sufficiently small from \eqref{v3}-\eqref{v7}.
\end{proof}

\begin{proof}[\bf{Proof of Theorem \ref{Th1.1}}]
From Lemmas \ref{lem2.3}-\ref{lem2.5}, Lemma \ref{lem3.8} and Mountain Pass Theorem, we get the existence of a nonnegative weak solution $u$ of \eqref{1.5}. 
Moreover, we can deduce that $u\in C^2(\overline{{{\mathbb{R}}^{N}_{+}}})$ from the Brezis-Kato Theorem and standard regularity theory for elliptic equations.
From maximum principle, $u$ is a positive solution of \eqref{1.5}, which means that $u$ is a positive solution of \eqref{1.1}.  The proof of Theorem \ref{Th1.1} is completed.
 \end{proof}

\vs{3mm}
\section{The proof of  Theorem   \ref{Th1.2}  }\label{S4}
\vs{2mm}

In this section, we prove the existence of multiple solutions of problem \eqref{1.1} by applying dual variational principle. To this end, we introduce some definitions and notations in the following.

\begin{Def}\label{def4.1}
Let $E$ be a Banach space. $B\subset E$ is called symmetric if $u\in B$ implies $-u\in B$. For a closed symmetric set $B$ which does not contain the origin, we define a genus $\nu(B)$ of $B$ by the smallest integer $k$ such that there exists an odd continuous mapping from $B$ to ${\mathbb{R}^{k}}\setminus \{0\}$. If there does not exist such a $k$, we define $\nu(B)=\infty$ and let~$\nu(\emptyset)=0$.
\end{Def}

Let $I\in C^1(E,\mathbb{R})$, $B_r$ be a ball in $E$ centered at $0$ with radius $r$, $\partial B_{r}$ be the boundary of $B_{r}$, and
\begin{eqnarray}
&&\Sigma:=\{B\subset E\setminus \{0\}: \ B~\mbox {is closed and symmetric} \},\notag\\
&&E_+:=\{u\in E: \ I(u)\geq  0\},\notag\\
&&H:=\{h: \ h\in C(E,E), h \mbox{ is an odd homeomorphism and}~h (B_1)\subset E_+\},\notag\\
&&\Gamma_{k}:=\{B\subset {\Sigma}: \ \mbox{B is compact},\nu(B\cap h({\partial {B_1}}))\geq k~\mbox{for any}~h\in H\}.\notag
\end{eqnarray}
Replacing $(PS)$ condition by $(PS)_c$ condition, we have the following Lemma proved exactly as in \cite{0}.

\begin{Lem}\label{lem4.2}
Assume $I\in C^1(E,\mathbb{R})$ satisfies the following properties:
\vskip 0.2cm

$(H_1)$  $I(0)=0, I(-u)=I(u)$ for all $u\in E$;
\vskip 0.2cm

$(H_2)$ {there exist}~${\alpha},\rho >0$~{such that}~$I(u)>0$ for any $u \in B_\rho\setminus \{0\}$, $I(u)\geq  \alpha$~{for all} ~$u\in {\partial {B_\rho}}$;
\vskip 0.2cm

$(H_3)$ for any finite dimensional subspace $E^{m}\subset E$, $E^{m}\cap E_+$ is bounded.
\vskip 0.2cm

For any $k=1, 2, \cdots,$ let
\begin{equation*}
b_k:=\inf\limits_{B\in \Gamma_k}{\sup\limits_{u\in B}}I(u),
\end{equation*}
then
\vskip 0.2cm
	
	$( \romannumeral 1)$  $\Gamma_k \neq\emptyset$ and $0<\alpha\leq b_k\leq b_{k+1} $;
\vskip 0.2cm
	
	$( \romannumeral 2)$  $b_k$ is a critical value if $I$ satisfies $(PS)_c$ condition for $c=b_k$.
\vskip 0.2cm
\noindent
Moreover, if $b=b_k=\cdots=b_{k+m}$, then $\nu(K_{b})\geq m+1$, where
$K_{b}=\{u\in E \ | \ I(u)=b, I'(u)=0 \}$.
\end{Lem}

In what follows, we take $E=X$ and use the same notations $\Sigma$, $B_{r}$, $\partial B_{r}$ and $\nu(B)$. Denote
\begin{eqnarray*}
&&E_{\lambda, \mu}:=\{u\in X: \ J_{\lambda, \mu}(u)\geq  0\},\\
&&E_{*}:=\{u\in X: \ J_{*}(u)\geq  0\},\\
&&H_{\lambda, \mu}:=\{h: \ h\in C(X,X), h \mbox{ is an odd homeomorphism and}~h (B_1)\subset E_{\lambda, \mu}\},\\
&&H_{*}:=\{h: \ h\in C(X, X), h \mbox{ is an odd homeomorphism and}~h (B_1)\subset E_{*}\},
\end{eqnarray*}
where
\begin{eqnarray*}
&&J_{\lambda, \mu}(u):={\frac{1}{2}}{\|{u}\|}^{2}-{\frac{\mu }{p}}{\|u\|}_{{L_{K}^{p}}({\mathbb{R}}^{N}_{+})}^{p}-{\frac{\lambda }{2^{*}}}\|u\|_{L^{{2}^{*}}_{K}({\mathbb{R}}^{N}_{+})}^{2^{*}}-{\frac{\sqrt{\lambda} }{2_{*}}}\|u\|_{L^{{2}_{*}}_{K}({\mathbb{R}}^{N-1})}^{2_{*}},\\
&&J_{*}(u):={\frac{1}{2}}{\|{u}\|}^{2}-{\frac{\mu }{p}}{\|u\|}_{{L_{K}^{p}}({\mathbb{R}}^{N}_{+})}^{p}.
\end{eqnarray*}
Obviously, $E_{\lambda, \mu}\subset E_{*}$ and  $H_{\lambda, \mu}\subset H_{*}$.

\begin{Lem}\label{lem4.3}
If $N\geq 3$, $p\in (2,2^{*})$, $\lambda>0$ and $\mu>0$, then $J_{\lambda, \mu}(u)$ and $ J_{*}(u)$ satisfy the properties $(H_{1})$,  $(H_{2})$ and $(H_{3})$.
\end{Lem}

\begin{proof}
Similar to the proof of Lemma \ref{lem2.3}, we easily verify $(H_{1})$ and $(H_{2})$. Thus, we just need to prove that $(H_{3})$ holds for $J_{\lambda, \mu}(u)$. We prove it by contradiction. If there exists a finite dimensional subspace $E^{m}\subset X$ such that $E^{m}\cap E_{\lambda, \mu}$ is unbounded, that is, there exists a sequence $\{u_n\}\subset E^{m}\cap E_{\lambda, \mu}$ such that ${\|{u_n}\|} \rightarrow \infty$ as $n\rightarrow \infty$. Let $e_1, e_2, \cdots, e_m$ be the orthonormal basis of $E^{m}$. For any $n\in \mathbb{N}$, there exists $a^{n}=(a_1^n, a_2^n, \cdots, a_m^n)\in {\mathbb{R}}^{m}$ such that
\begin{equation*}
u_n=a_1^n e_1+a_2^n e_2+\cdots+a_m^n e_m.
\end{equation*}

Since ${\|{u_n}\|}=\Big(\sum\limits_{i=1}^m |a_i^n|^{2}\Big)^{\frac{1}{2}} \rightarrow \infty$ as $n\rightarrow \infty$, one has $a_n \rightarrow \infty$ as $n\rightarrow \infty$, where $a_n:=\max\limits_{1\leq i\leq m}|a_i^n|$. Clearly,
\begin{equation}\label{4.1}
{\|{u_n}\|}^{2}=O(|a_n|^2).
\end{equation}	
From the norm equivalence property of finite dimensional space, we infer that
there exist some constants $c_1$, $c_2>0$ such that
\begin{equation}\label{4.2}
{\|u_n\|}_{{L_{K}^{p}}({\mathbb{R}}^{N}_{+})}^{p}\geq c_1{\|{u_n}\|}^{p}\geq c_2|a_n|^p.
\end{equation}

Combining \eqref{4.1}, \eqref{4.2} with $p\in (2, 2^{*})$, $\lambda$, $\mu>0$, we deduce that for $n$ sufficiently large,
\begin{equation*}
J_{\lambda, \mu}(u_n)< J _{*}(u_n)\leq O(|a_n|^2)-{\frac{\mu c_2 }{p}}|a_n|^p<0,
\end{equation*}
which contradicts $u_n\in E_{\lambda, \mu}$. Thus, $(H_{3})$ holds for $J_{\lambda, \mu}(u), J_{*}(u)$.
\end{proof}
For any $k=1, 2, \cdots$, define
\begin{eqnarray*}
&&\Gamma_{\lambda, \mu}^{k}:=\{B\subset {\Sigma}: \ \mbox{B is compact},\nu(B\cap h({\partial {B_1}}))\geq k~\mbox{for any}~h\in H_{\lambda, \mu}\},\notag\\
&&\Gamma_{*}^{k}:=\{B\subset {\Sigma}: \ \mbox{B is compact},\nu(B\cap h({\partial {B_1}}))\geq k~\mbox{for any}~h\in H_{*}\}, \notag\\
&&c_{\lambda, \mu}^k:=\inf\limits_{B\in \Gamma_{\lambda, \mu}^{k}}{\sup\limits_{u\in B}}J_{\lambda, \mu}(u),~~
c_{*}^k:=\inf\limits_{B\in \Gamma_{*}^{k}}{\sup\limits_{u\in B}}J_{*}(u).
\end{eqnarray*}
It is clear to check that $\Gamma_{*}^{k}\subset\Gamma_{\lambda, \mu}^k$ for $k=1, 2, \cdots$. Moreover, we conclude $\Gamma_{*}^k \neq\emptyset$ and $0<\alpha\leq c_{*}^k<\infty$ from Lemma \ref{lem4.2}.

\begin{proof}[\bf{Proof of Theorem \ref{Th1.2}}]
For any $k=1, 2, \cdots$, by Lemma \ref{lem4.2} and the definitions of $c_{\lambda, \mu}^k$, $c_{*}^k$, we obtain that for any $j=1, 2, \cdots, k$,
\begin{equation}\label{4.3}
	\begin{aligned}
c_{\lambda, \mu}^j
&\leq c_{\lambda, \mu}^k
=\inf\limits_{B\in \Gamma_{\lambda, \mu}^{k}}{\sup\limits_{u\in B}}J_{\lambda, \mu}(u)\leq\inf\limits_{B\in \Gamma_{*}^{k}}{\sup\limits_{u\in B}}J_{\lambda, \mu}(u)\\
&\leq\inf\limits_{B\in \Gamma_{*}^{k}}{\sup\limits_{u\in B}}J_{*}(u)=c_{*}^k.
	\end{aligned}
\end{equation}

Next, we claim that for each $j=1, 2, \cdots, k$, $\lambda\in (0,\lambda_k)$, $J_{\lambda, \mu}$ satisfies $(PS)_c$ condition for $c=c_{\lambda, \mu}^j$. In fact, taking $\lambda_{k}:=(\frac{A}{c_*^k})^{\frac{2}{N-2}}$, from \eqref{2000} and \eqref{4.3}, we obtain
\begin{equation*}
c_{\lambda, \mu}^j\leq{c_*^k}=\lambda_{k}^{-\frac{N-2}{2}}A<
\lambda^{-\frac{N-2}{2}}A=A_\lambda.
\end{equation*}By Lemma \ref{lem2.5}, $J_{\lambda, \mu}$ satisfies the $(PS)_c$ condition for $c=c_{\lambda, \mu}^j$, $j=1, 2, \cdots, k$. From Lemmas \ref{lem4.2}-\ref{lem4.3}, $J_{\lambda, \mu}(u)$ has at least $k$ different critical points $u_{j}\in X$ such that $J_{\lambda, \mu}(u_{j})=c_{\lambda, \mu}^j$ with $j=1, 2, \cdots, k$. Since the functional $J_{\lambda, \mu}$ is even, $-u_{j}$ is also a critical point. Therefore, $\pm u_{j}$, $j=1, 2, \cdots, k$, are solutions of problem (\ref{1.5}), which means that $\pm u_{j}$ are solutions of problem (\ref{1.1}).
We complete the proof of this Theorem.
\end{proof}




\end{document}